\documentclass[12pt]{amsart} 
\title[Koornwinder moments and the two-species exclusion process]{Macdonald-Koornwinder moments and the two-species exclusion process}
\author{Sylvie Corteel and Lauren K. Williams}
\date{\today}
\thanks{SC was partially funded by the ``Combinatoire \`a Paris" projet 
Emergences 2013--2017 and by ``ALEA Sorbonne" projet IDEX USPC.
LW was partially supported by 
the Fondation Sciences Math\'ematiques de Paris, the Simons foundation,
a Rose-Hills Investigator award, and an NSF CAREER award.  
Both authors are grateful for the comments of the 
anonymous referees, and the support of the France-Berkeley fund,
and would like to acknowledge the hospitality of LIAFA, where part of this 
work was carried out.}
\address{Laboratoire d'Informatique Algorithmique: Fondements et Applications,
Centre National de la Recherche Scientifique et Universit\'e Paris Diderot,
Paris 7, Case 7014, 75205 Paris Cedex 13
France} 
\email{corteel@liafa.univ-paris-diderot.fr}
\address{Department of Mathematics, University of California, Berkeley,
Evans Hall Room 913, Berkeley, CA 94720}
\email{williams@math.berkeley.edu}



\usepackage{pstricks, pst-node, amssymb}
\usepackage{bbm}

\psset{unit=1pt, arrowsize=4pt, linewidth=.7pt}
\psset{linecolor=blue}
\newgray{grayish}{.90}
\newrgbcolor{embgreen}{0 .5 0}
\def\dd{\mathsf{d}}
\def\ee{\mathsf{e}}

\def\vblack(#1, #2)#3{\cnode*[linecolor=black](#1, #2){3}{#3}}
\def\vwhite(#1,#2)#3{\cnode[linecolor=black,fillcolor=white,fillstyle=solid](#1,
#2){3}{#3}}
\countdef\x=23
\countdef\y=24
\countdef\z=25
\countdef\t=26

\def\tbox(#1,#2)#3{
\x=#1 \y=#2
\multiply\x by 12
\multiply\y by 12
\z=\x \t=\y
\advance\z by 12
\advance\t by 12
\psline(\x,\y)(\x,\t)(\z,\t)(\z,\y)(\x,\y)
\advance\x by 6
\advance\y by 6
\rput(\x,\y){{\bf #3}}}

\usepackage{amsmath, amsthm, amssymb, amsbsy}
\usepackage{amsfonts, latexsym, stmaryrd, amscd, xy}
\usepackage[mathscr]{eucal}
\usepackage{epsfig}

\newtheorem{theorem}{Theorem}[section]

\newtheorem{proposition}[theorem]{Proposition}
\newtheorem{lemma}[theorem]{Lemma}

\newtheorem{corollary}[theorem]{Corollary}
\newtheorem{conjecture}[theorem]{Conjecture}
\newtheorem{remark}[theorem]{Remark}

\newtheorem{definition}[theorem]{Definition}

\usepackage{amsfonts}
\usepackage{xy}
\usepackage{amssymb}
\usepackage{amsmath}

\newcommand{\ttt}{\tau}

\newcommand{\Z}{\mathcal{Z}}
\newcommand{\K}{\mathcal{K}}

\newcommand{\zz}{\mathbf z}
\newcommand{\C}{\mathcal C}

\newcommand{\D}{\mathcal{D}}
\newcommand{\DD}{{D}}
\newcommand{\EE}{{E}}

\newcommand{\E}{\mathcal{E}}

\newcommand{\indicator}{\mathbbm{1}}

\DeclareMathOperator{\const}{const}
\DeclareMathOperator{\Motz}{Motz}
\DeclareMathOperator{\denom}{den}
\DeclareMathOperator{\den}{den}
\DeclareMathOperator{\num}{num}

\DeclareMathOperator{\inv}{inv}

\def\llangle{\langle}
\def\rrangle{\rangle}

\newcommand{\thmrefer}[1]{\renewcommand\thetheorem
  {\protect\ref{#1}}\addtocounter{theorem}{-1}}

\xyoption{all}
\CompileMatrices

\begin{document}



\begin{abstract}
Introduced in the late 1960's \cite{bio, Spitzer},
the asymmetric exclusion process (ASEP) is an important model
from statistical mechanics which describes a system of interacting
particles hopping left and right on a one-dimensional lattice 
with open boundaries.
It has been known for awhile
that there is a tight connection between the 
partition function of the ASEP and moments of 
Askey-Wilson polynomials \cite{USW, CW-Duke1, CSSW}, a family 
of orthogonal polynomials which are at the top of the hierarchy of 
classical orthogonal polynomials in one variable.    On 
the other hand, Askey-Wilson polynomials can be viewed as a specialization
of the multivariate \emph{Macdonald-Koornwinder polynomials}
 (also known as Koornwinder polynomials), which in turn give
rise to 
the Macdonald polynomials associated to any classical root system
via a limit or specialization
\cite{vanDiejen}. 
In light of the fact that Koornwinder polynomials 
generalize the Askey-Wilson polynomials,
it is natural to ask whether
one can find a particle model whose partition function is related to 
Koornwinder polynomials.  In this article we answer this question
affirmatively, by showing that Koornwinder moments at $q=t$
are closely connected to the partition function for the 
\emph{two-species  exclusion process}.

\end{abstract}

\maketitle

\setcounter{tocdepth}{1}
\tableofcontents

\section{Introduction}

Introduced in the late 1960's \cite{bio, Spitzer},
the asymmetric exclusion process (ASEP) is a model
of interacting
particles hopping left and right on a one-dimensional lattice of $N$
sites.  
In the most general form of the ASEP with open boundaries,
particles may enter and exit at the left with probabilities
$\alpha$ and $\gamma$, and they may exit and enter at the right
with probabilities $\beta$ and $\delta$.  In the bulk,
the probability of hopping left is $q$ times the probability of hopping
right.
The ASEP is important in statistical mechanics because
it is one of the simplest models which exhibits 
boundary-induced phase transitions.  Moreover, the ASEP 
has been cited as a model for traffic flow and
protein synthesis.  

It has been known since work of Uchiyama-Sasamoto-Wadati \cite{USW} that there
is a close connection between the partition function $Z_N$ of the ASEP, and Askey-Wilson polynomials,
a family 
of orthogonal polynomials which are at the top of the hierarchy of 
classical orthogonal polynomials in one variable.  Using 
their work, we showed in \cite{CSSW} that  
each Askey-Wilson moment equals a specialization of the \emph{fugacity partition function} $Z_N(\xi)$
of the ASEP, see Theorem \ref{thm:ZN-AW}.
In \cite{CW-Duke1, CW-Duke2}, we introduced some new combinatorial objects called
\emph{staircase tableaux}, and used them to completely describe the stationary distribution
of the ASEP; in particular, the partition function $Z_N(\xi)$ can be written as a sum over
staircase tableaux of size $N$.  It follows that (up to a scalar factor), $Z_N(\xi)$ is a 
polynomial with positive coefficients, and that 
Askey-Wilson moments can be expressed in terms of staircase tableaux.

Askey-Wilson polynomials can be viewed as a specialization of 
the multivariate \emph{Macdonald-Koornwinder polynomials}, which are 
also known as \emph{Koornwinder polynomials} \cite{Koornwinder},
or \emph{Macdonald polynomials for the type BC root system} \cite{Macdonald}.
(For brevity, we will henceforth call them Koornwinder polynomials.)
These polynomials
are particularly important because 
the Macdonald polynomials associated to any classical root system
can be expressed as limits or special cases of Koornwinder 
polynomials 
\cite{vanDiejen}.   Since Koornwinder polynomials generalize 
Askey-Wilson polynomials, and Askey-Wilson moments are closely connected to the ASEP,
it is natural to ask if there is some particle model generalizing the ASEP whose 
partition function is related to Koornwinder polynomials.  This question was posed to us 
by Mark Haiman in 2007 \cite{Haiman}.

When $q=t$, Koornwinder polynomials can be expressed in terms of Askey-Wilson polynomials
by means of a Schur-like determinantal formula.
This fact, together with the connection between Askey-Wilson moments and the partition 
function $Z_N(\xi)$ of the ASEP, led Eric Rains \cite{Rains} to suggest that we consider,
for any partition $\lambda = (\lambda_1,\lambda_2,\dots, \lambda_m)$, the quantity 
\begin{equation}\label{Kmoments}
K_{\lambda}(\xi) = \frac{\det(Z_{\lambda_i+m-i+m-j}(\xi))_{i,j=1}^m}{\det(Z_{2m-i-j}(\xi))_{i,j=1}^m}.
\end{equation}
The $K_{\lambda}(\xi)$'s can be considered to be Koornwinder moments,
see Sections \ref{sec:Koornwinder} and 
\ref{sec:moments2}.

In this paper we will answer Haiman's question affirmatively by 
demonstrating that there is a close connection between 
Koornwinder moments and the partition function of the \emph{two-species ASEP}.  
The two-species ASEP is a generalization of the ASEP which involves two different
kinds of particles, ``heavy" and ``light".  Both types of particles can hop left and right 
in the lattice (heavy and light particles interact exactly as do particles and holes
in the usual ASEP), but only the heavy particles can enter and exit the lattice
at the left and right boundary.  So in particular, the number of light particles
is conserved.  When there are no light particles, the two-species
ASEP reduces to the usual ASEP.
The  main result of this paper is 
an interpretation of the ``complete homogeneous"  Koornwinder moments 
$K_{(m,0,0,\dots,0)}(\xi)$ in terms of the partition function of the 
two-species ASEP.  That is, we prove the following.
\begin{theorem}\label{thm:main}
The Koornwinder moment $K_{(N-r,0,0,\dots,0)}(\xi)$ (where there are precisely $r$ $0$'s in the partition)
is proportional to the fugacity partition function $Z_{N,r}(\xi)$ for 
the two-species ASEP on a lattice of $N$ sites with $r$ ``light" particles.
More specifically, 
$$K_{(N-r,0,0,\dots,0)}(\xi) = 
	Z_{N,r}(\xi).$$ 
\end{theorem}

As an intermediate step towards proving Theorem \ref{thm:main},
we give a combinatorial interpretation of $Z_{N,r}(\xi)$ in terms
of partial Motzkin paths, see Theorem 
\ref{thm:main2}.  This leads to an integral representation 
for $Z_{N,r}(\xi)$, see Corollary \ref{cor:integral}, which in turn
can be used to compute asymptotics.

We also prove a Jacobi-Trudi-type formula which expresses any arbitrary Koornwinder moment
$K_{\lambda}$ as a determinant in the complete homogeneous Koornwinder moments, see Corollary \ref{thm:JT}.

Since the partition function $Z_N$ is a polynomial with positive coefficients
\cite{CW-Duke1, CW-Duke2}, Rains conjectured that the Koornwinder moments
$K_{\lambda}$ are also
polynomials with positive coefficients (up to a simple scalar factor) \cite{Rains}.
Note that by the probabilistic interpretation of Theorem 
\ref{thm:main}, it follows that 
$K_{(N-r,0,0,\dots,0)}$ is positive whenever we specialize the parameters
$\alpha, \beta, \gamma, \delta$, and $q$ to be positive numbers between
$0$ and $1$.
Moreover, in a sequel to this paper 
\cite{CMW}, 
we will 
give a tableaux formula for $Z_{N,r}$, which implies that both
$Z_{N,r}$ and $K_{(N-r,0,0,\dots,0)}$ are  (up to a scalar factor) polynomials
in the parameters with positive coefficients.

It is worth noting that there has been some recent work which may be related to ours.
In \cite{BorodinCorwin}, Borodin and Corwin introduced what they call
\emph{Macdonald processes}, which are 
probability measures on sequences of partitions defined in terms of nonnegative specializations 
of the (type A) Macdonald symmetric functions, and showed that they are related to various interacting particle systems.
None of the particle systems they discuss is the ASEP (or two-species ASEP) with open boundaries, but perhaps there exists
some more general \emph{Koornwinder processes} which would be connected to the ASEP.

Additionally, while we were writing up this paper, we learned about some work of Cantini
\cite{Cantini}, which contains a result similar to our 
Theorem  \ref{thm:main}.  However, Cantini uses the partition
$(1^{N-r},0^r)$, and his techniques are completely different from ours; 
e.g. he uses the affine Hecke algebra of type $\hat C_N$ as opposed
to the Matrix Ansatz and the combinatorics of Motzkin paths.

The structure of this paper is as follows.
In Section \ref{sec:Koornwinder}, we give an introduction
to Askey-Wilson polynomials, and Koornwinder polynomials,
and Koornwinder moments at $q=t$.
In Section \ref{sec:ASEP} we define the asymmetric exclusion 
process and the two-species exclusion process.  We also explain the 
\emph{Matrix Ansatz}, which is a powerful tool for analyzing the 
stationary distribution of these models.  In Section \ref{sec:Motzkin},
we give an interpretation of Koornwinder moments in terms of 
weighted Motzkin paths.
In Section \ref{sec:JT} we state and prove a Jacobi-Trudi type 
formula for Koornwinder moments.  In Section \ref{sec:M-2}
we prove a connection between partial Motzkin paths and the 
two-species exclusion process.  In Section \ref{sec:thm}, 
we complete the proof of our main result.  Finally in
Section \ref{sec:specialize} we show that when we specialize
$\xi=q=1$, the Koornwinder moments $K_{\lambda}$ have a beautiful multiplicative
formula in terms of the hook lengths of the corresponding partition.  
This formula provides evidence for the 
positivity conjecture regarding Koornwinder moments.

\textsc{Acknowledgments:} 
We would like to thank Mark Haiman, who pointed out to us in 2007 that 
Koornwinder polynomials generalize the Askey-Wilson polynomials, and asked us 
if we could make a connection between Koornwinder polynomials and some 
generalization of the ASEP.  We would also like to thank Eric Rains, 
who suggested that we look at Koornwinder polynomials for $q=t$, and 
in particular at the quantities defined by \eqref{Kmoments}.
Finally we would like to thank Jennifer Morse and Dennis Stanton for useful comments, and 
Donghyun Kim for noticing a typo in Theorem \ref{thm:main} in the published version of the paper (which is 
corrected in this version).

\section{Askey-Wilson polynomials and Koornwinder polynomials 
at $q=t$}\label{sec:Koornwinder}

The Askey-Wilson polynomials are  orthogonal polynomials
with five free parameters ($a, b, c, d, q$).
They reside at the top of the hierarchy of the one-variable
orthogonal polynomial family in the Askey scheme \cite{AW,GR,Koekoek}.
In this section we define  Askey-Wilson polynomials, following
the exposition of \cite{AW} and \cite{USW}, as well as Askey-Wilson moments.
We will then define Koornwinder polynomials and their moments for $q=t$.

The $q$-shifted factorial is defined by
\begin{eqnarray*}
(a_1,a_2,\cdots,a_s;q)_n=\prod_{r=1}^s \prod_{k=0}^{n-1} (1-a_rq^k),
\end{eqnarray*}
and the basic hypergeometric function is given by
\begin{eqnarray*}
{}_r\phi_s\left[ {{a_1,\cdots ,a_r}\atop{b_1,\cdots ,b_s}};q,z\right]
=\sum_{k=0}^\infty \frac{(a_1,\cdots,a_r;q)_k}{(b_1,\cdots,b_s,q;q)_k}
((-1)^k q^{k(k-1)/2})^{1+s-r} z^k.
\end{eqnarray*}

The Askey-Wilson polynomial $P_n(x)=P_n(x;a,b,c,d\vert q)$
is explicitly defined by
\begin{eqnarray*}
P_n(x)=
a^{-n}(ab,ac,ad;q)_n\
{}_4\phi_3\left[ {{q^{-n},q^{n-1}abcd,ae^{i\theta},ae^{-i\theta}}
        \atop{ab,ac,ad}};q,q \right] ,
\label{eqn:defAW}
\end{eqnarray*}
with $x=\cos\theta$ for $n\in \Z_+:=\{0,1,2,\cdots\}$.
It satisfies the three-term recurrence
\begin{eqnarray*}
A_nP_{n+1}(x)+B_nP_n(x)+C_nP_{n-1}(x)=2xP_n(x),
\label{eqn:recAW}
\end{eqnarray*}
with $P_0(x)=1$ and $P_{-1}(x)=0$,
where
\begin{align*}
A_n&=
\frac{1-q^{n-1}abcd}{(1-q^{2n-1}abcd)(1-q^{2n}abcd)},
\\
B_n&=
\frac{q^{n-1}}{(1-q^{2n-2}abcd)(1-q^{2n}abcd)}
[(1+q^{2n-1}abcd)(qs+abcds')-q^{n-1}(1+q)abcd(s+qs')],
\\
C_n&=
\frac{(1-q^n)(1-q^{n-1}ab)(1-q^{n-1}ac)(1-q^{n-1}ad)(1-q^{n-1}bc)
(1-q^{n-1}bd)(1-q^{n-1}cd)}
{(1-q^{2n-2}abcd)(1-q^{2n-1}abcd)},
\end{align*}
\begin{eqnarray*}
\text{ and }~~s=a+b+c+d, \qquad s'=a^{-1}+b^{-1}+c^{-1}+d^{-1}.
\end{eqnarray*}

\begin{remark}\label{AWSymmetry}
It is obvious from the three-term recurrence that the polynomials
$P_n(x)$ are symmetric in $a, b, c$ and $d$.
\end{remark}

For $|a|, |b|, |c|, |d| < 1$, using
$z=e^{i\theta}$,
the orthogonality is expressed by
\begin{eqnarray*}
\oint_C \frac{dz}{4\pi iz} w\left(\frac{z+z^{-1}}{2}\right)
P_m\left(\frac{z+z^{-1}}{2}\right)P_n\left(\frac{z+z^{-1}}{2}\right)
=\frac{h_n}{h_0}\delta_{mn},
\label{eqn:orthoointAW}
\end{eqnarray*}
where the integral contour $C$ is a closed path which
encloses the poles at $z=aq^k$, $bq^k$, $cq^k$, $dq^k$ $(k\in \Z_+)$
and excludes the poles at $z=(aq^k)^{-1}$, $(bq^k)^{-1}$, $(cq^k)^{-1}$,
$(dq^k)^{-1}$ $(k\in \Z_+)$, and where
\begin{eqnarray*}
&&w(\cos\theta)=
\frac{(e^{2i\theta},e^{-2i\theta};q)_\infty}
{(ae^{i\theta},ae^{-i\theta},be^{i\theta},be^{-i\theta},
ce^{i\theta},ce^{-i\theta},de^{i\theta},de^{-i\theta};q)_\infty}, \\
&&\frac{h_n}{h_0}=
\frac{(1-q^{n-1}abcd)(q,ab,ac,ad,bc,bd,cd;q)_n}
{(1-q^{2n-1}abcd)(abcd;q)_n} ,\\
&&h_0=
\frac{(abcd;q)_\infty}{(q,ab,ac,ad,bc,bd,cd;q)_\infty} .
\end{eqnarray*}
(In the other parameter region, the orthogonality is continued analytically.)

\begin{remark}
Note that our definition of the weight function above differs slightly from 
the definition given in \cite{AW}; the weight function in \cite{AW} 
did not have the $h_0$ in the denominator.  Our convention simplifies some of 
the formulas to come.
\end{remark}

\begin{definition}\label{def:AW-moments}
The moments of the (weight function of the) Askey-Wilson polynomials
-- which we sometimes refer to as simply the \emph{Askey-Wilson moments} --
are defined by
\begin{eqnarray*}
\mu_k = \mu_k(a,b,c,d|q) = \oint_C \frac{dz}{4\pi iz} w\left(\frac{z+z^{-1}}{2}\right)
\left(\frac{z+z^{-1}}{2}\right)^k.
\label{eqn:moment}
\end{eqnarray*}
\end{definition}



\begin{definition}\label{KoornwinderPolynomials}
Let $\zz=(z_1,\dots, z_m)$, $\lambda = (\lambda_1,\dots, \lambda_m)$ 
be a partition,
and $a, b, c, d, q, t$ be generic complex parameters.
The \emph{Koornwinder polynomials} $P_{\lambda}(\zz; a, b, c, d|q,t)$ are
multivariate orthogonal polynomials which are the type BC-case of Macdonald polynomials.
More specifically, 
$P_{\lambda}(\zz; a, b, c, d|q,t)$ is the unique Laurent polynomial which is 
invariant under permutation and inversion of variables, with leading
monomial $\zz^{\lambda}$, and orthogonal with respect to the 
\emph{Koornwinder density}
$$ \prod_{1 \leq i < j \leq m}
\frac{(z_i z_j, z_i/z_j, z_j/z_i, 1/{z_i z_j}; q)_{\infty}}
{(t z_i z_j, tz_i/z_j, tz_j/z_i, t/{z_i z_j};q)_{\infty}}
\prod_{1 \leq i \leq m} 
\frac{(z_i^2, 1/{z_i^2}; q)_{\infty}}
{(az_i, a/z_i, bz_i, b/z_i, cz_i, c/z_i, dz_i, d/z_i;q)_{\infty}}$$
on the unit torus 
$|z_1|=|z_2| = \dots = |z_m| = 1$, where the parameters satisfy
$|a|, |b|, |c|, |d|, |q|,|t|<1$.

At $q=t$, we have
$$P_{\lambda}(\zz; a, b, c, d| q, q) = \const \cdot
 \frac{\det(p_{m-j+\lambda_j}(z_i; a, b, c, d| q))_{i,j=1}^m}
{\det(p_{m-j}(z_i; a, b, c, d|q))_{i,j=1}^m},$$
where the $p_i$'s are the Askey-Wilson polynomials.

Note that when $q=t$, the Koornwinder density becomes 
\begin{equation}\label{eq:K}
\prod_{1\le i<j\le m} (1-z_iz_j)(1-z_i/z_j)(1-z_j/z_i)(1-1/z_iz_j)
\prod_{1\le i\le m} w\left(\frac{z_i+z_i^{-1}}{2}\right),
\end{equation}
where $w$ denotes the Askey-Wilson density.
\end{definition}

\begin{remark}
When $q=t$, Koornwinder polynomials are sometimes called 
\emph{Macdonald's 9th variation of Schur functions associated with Askey-Wilson polynomials} \cite{Noumi}.
\end{remark}

For Askey-Wilson polynomials, the $k$th moment $\mu_k$ is defined
to be the integral of $x^k$ (here $x = \frac{z+z^{-1}}{2}$) with respect 
to the Askey-Wilson density.  For the multivariate Koornwinder polynomials, 
there are several ways that we could define moments.  One way would be to 
integrate a monomial in $x_1,\dots, x_m$ (here we set 
$x_i = \frac{z_i+z_i^{-1}}{2}$) with respect to the Koornwinder density
\eqref{eq:K}.
Following a suggestion of Eric Rains \cite{Rains}, we will instead 
define our Koornwinder moments by integrating 
Schur polynomials $s_{\lambda}(x_1,\dots,x_m)$ with respect to \eqref{eq:K}.

\begin{definition}
Let $I_k(f(x_1,\dots,x_m);a,b,c,d;q,q)$ denote the result of integrating the 
function $f(x_1,\dots,x_m)$ with respect to the Koornwinder density 
\eqref{eq:K}.
We define the \emph{Koornwinder moment}
$$M_{\lambda}=M_{\lambda}(a,b,c,d | q) 
= I_k(s_{\lambda}(x_1,\dots,x_m); a, b, c, d; q,q).$$
\end{definition}

As Rains pointed out to us, these Koornwinder moments have a 
determinantal formula \cite{Rains}.

\begin{lemma}\label{lem:Koornwindermoment}
We have that 
$$M_{\lambda}
= \frac{\det(\mu_{\lambda_i+m-i+m-j})_{i,j=1}^m}
{\det(\mu_{2m-i-j})_{i,j=1}^m},$$ where $\mu_k$ is an Askey-Wilson moment.
\end{lemma}

\begin{proof}
Note that when $q=t$, the Koornwinder density can be written (up to 
a scalar factor of $2^{m(m-1)}$) as 
\begin{equation}
\prod_{1\le i<j\le m} (x_i-x_j)^2
\prod_{1 \leq i \leq m} w(x_i),
\end{equation}
where $x_i = \frac{z_i+z_i^{-1}}{2}$.

Recall that the classical definition of the Schur polynomials 
says that 
$$s_{\lambda}(x_1,\dots, x_m) = 
\frac{\det(x_j^{\lambda_i+m-i})_{i,j = 1}^m}
{\det(x_j^{m-i})_{i,j=1}^m},$$ so that 
$$s_\lambda(x_1,\dots,x_m)\prod_{1\le i<j\le m} (x_i-x_j)^2
=
\det_{1\le i,j\le m}(x_j^{\lambda_i+m-i})
\det_{1\le i,j\le m}(x_j^{m-i}).$$

Therefore we have that 
\begin{align*}
M_{\lambda}
& \propto
\int
\det_{1\le i,j\le m}(x_j^{\lambda_i+m-i})
\det_{1\le i,j\le m}(x_j^{m-i})
\prod_{1\le i\le m} w(x_i)\\
& \propto 
\det_{1\le i,j\le m}
  \int x^{\lambda_i+m-i} x^{m-j} w(x)\\
&=
\det_{1\le i,j\le m} \mu_{\lambda_i+m-i+m-j}.
\end{align*}
Here we obtained the second line by applying the 
integral version of the Cauchy-Binet formula 
\cite{deBruijn, Andreief},
and we obtained the third line by using the definition of 
Askey-Wilson moments.

Since the constant of proportionality is independent of $\lambda,$ 
we can recover it by setting $\lambda=0,$ and using
$I_k(1) = 1$.  This gives us 
$$I_K(s_\lambda(x_1,\dots, x_m);a,b,c,d;q,q)
=
\frac{\det(\mu_{\lambda_i+m-i+m-j})_{i,j = 1}^m}
{\det (\mu_{2m-i-j})_{i,j=1}^m}.$$
\end{proof}

\section{The asymmetric exclusion process and the two-species exclusion process}\label{sec:ASEP}

We start by defining the asymmetric exclusion process (ASEP)
with open boundaries.  We will then define the two-species exclusion process,
which generalizes the usual ASEP.  Finally we will explain the Matrix Ansatz, which 
has been an important tool for analyzing the stationary distribution of the ASEP.

\subsection{The asymmetric exclusion process (ASEP)}

\begin{definition}
Let $\alpha$, $\beta$, $\gamma$, $\delta$,  $q$, and $u$ be constants such that 
$0 \leq \alpha \leq 1$, $0 \leq \beta \leq 1$, 
$0 \leq \gamma \leq 1$, $0 \leq \delta \leq 1$, 
$0 \leq q \leq 1$,
and $0 \leq u \leq 1$.
Let $B_N$ be the set of all $2^N$ words in the
language $\{\circ, \bullet\}^*$.
The \emph{ASEP} is the Markov chain on $B_N$ with
transition probabilities:
\begin{itemize}
\item  If $X = A\bullet \circ B$ and
$Y = A \circ \bullet B$ then
$P_{X,Y} = \frac{u}{N+1}$ (particle hops right) and
$P_{Y,X} = \frac{q}{N+1}$ (particle hops left).
\item  If $X = \circ B$ and $Y = \bullet B$
then $P_{X,Y} = \frac{\alpha}{N+1}$ (particle enters from the left).
\item  If $X = B \bullet$ and $Y = B \circ$
then $P_{X,Y} = \frac{\beta}{N+1}$ (particle exits to the right).
\item  If $X = \bullet B$ and $Y = \circ B$
then $P_{X,Y} = \frac{\gamma}{N+1}$ (particle exits to the left).
\item  If $X = B \circ$ and $Y = B \bullet$
then $P_{X,Y} = \frac{\delta}{N+1}$ (particle enters from  the right).
\item  Otherwise $P_{X,Y} = 0$ for $Y \neq X$
and $P_{X,X} = 1 - \sum_{X \neq Y} P_{X,Y}$.
\end{itemize}
\end{definition}

See Figure \ref{states} for an
illustration of the four states, with transition probabilities,
for the case $N=2$.  The probabilities on the loops
are determined by the fact that the sum of the probabilities
on all outgoing arrows from a given state must be $1$.  

\begin{figure}[h]
\centering
\includegraphics[height=1.5in]{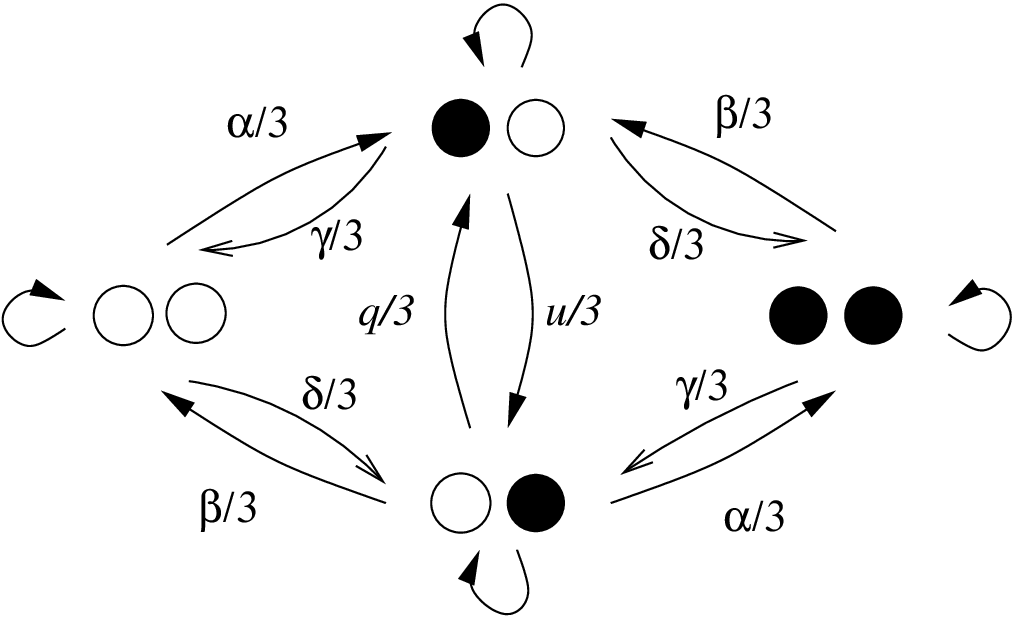}
\caption{The state diagram of the ASEP for $N=2$}
\label{states}
\end{figure}

In the long time limit, the system reaches a steady state where all 
the probabilities $P_n(\ttt_1, \ttt_2, \dots , \ttt_N)$ of finding
the system in configurations $(\ttt_1, \ttt_2, \dots , \ttt_N)$ are
stationary.  More specifically, the {\it stationary distribution}
is the unique (up to scaling) eigenvector of the transition
matrix of the Markov chain with eigenvalue $1$.

\subsection{The two-species exclusion process}

The two-species ASEP is a generalization of the ASEP which involves two kinds of 
particles, \emph{heavy} and \emph{light}.  We will denote a heavy particle by a $2$
and a light particle by a $1$.  We will also denote a hole (or the absence of a particle)
by a $0$.  In the two-species ASEP, heavy particles behave exactly as do 
particles in the usual ASEP: they can hop in and out of the boundary, and they can 
hop left and right in the lattice (swapping places with a hole or with a light particle).
Light particles cannot hop in and out of the boundary, but they may hop
left and right in the lattice (swapping places with a hole or with a heavy particle).
Therefore the number $r$ of light particles is conserved.

\begin{definition}
Let $\alpha$, $\beta$, $\gamma$, $\delta$,  $q$, and $u$ be constants such that 
$0 \leq \alpha \leq 1$, $0 \leq \beta \leq 1$, 
$0 \leq \gamma \leq 1$, $0 \leq \delta \leq 1$, 
$0 \leq q \leq 1$,
and $0 \leq u \leq 1$.
Let $B_{N,r}$ be the set of all words in 
$\{0,1,2\}^N$ which contain precisely $r$ $1$'s;
note that $|B_{N,r}| = {N \choose r} 2^{N-r}$.
The \emph{two-species ASEP} is the Markov chain on $B_{N,r}$ with
transition probabilities:
\begin{itemize}
\item  If $X = A 2 1 B$ and 
$Y = A 1 2 B$, or  
if $X = A 2 0 B$ and $Y = A 0 2 B$, or 
if $X = A 1 0 B$ and $Y = A 0 1 B$, 
then
$P_{X,Y} = \frac{u}{N+1}$  and
$P_{Y,X} = \frac{q}{N+1}$.
\item  If $X = 0 B$ and $Y = 2 B$
then $P_{X,Y} = \frac{\alpha}{N+1}$.
\item  If $X = B 2$ and $Y = B 0$
then $P_{X,Y} = \frac{\beta}{N+1}$.
\item  If $X = 2 B$ and $Y = 0 B$
then $P_{X,Y} = \frac{\gamma}{N+1}$.
\item  If $X = B 0$ and $Y = B 2$
then $P_{X,Y} = \frac{\delta}{N+1}$.
\item  Otherwise $P_{X,Y} = 0$ for $Y \neq X$
and $P_{X,X} = 1 - \sum_{X \neq Y} P_{X,Y}$.
\end{itemize}
\end{definition}

Note that if $r=0$, i.e. there are no light particles,
then the two-species ASEP is simply the usual ASEP.


\subsection{The Matrix Ansatz}

We now explain the \emph{Matrix Ansatz},
a technique introduced in \cite{DEHP} for computing the stationary 
distribution of the ASEP.  We also present Uchiyama's generalization of the Matrix Ansatz
\cite{Uchiyama} to the two-species case.  
Finally we will recall the solution to the Matrix Ansatz
which was found in \cite{USW}.

For convenience, we now set $u=1$.  Derrida, Evans, Hakim, and Pasquier
\cite{DEHP} proved the following theorem.

\begin{theorem}\cite{DEHP} \label{ansatz}
Suppose that there are matrices $D$ and $E$,  
and vectors $\langle W|$ and $|V\rangle$ such that the following relations hold:
\begin{itemize}
\item $\langle W| (\alpha E-\gamma D) = \langle W|$
\item $(\beta D-\delta E)|V\rangle = |V\rangle$
\item  $DE-qED = D+E$
\end{itemize}
Let $Z_{N} = {\langle W| (D+E)^N |V \rangle}$.
Then in the ASEP on a lattice of $N$ sites, 
the steady state probability of state $(\tau_1,\dots, \tau_N)$ is equal to 
$$\frac{\langle W| \prod_{i=1}^N [\indicator(\tau_i=\bullet)D+\indicator(\tau_i=\circ)E]|V\rangle}{Z_{N}}.$$
\end{theorem}

For example, the steady state probability of state 
$(\bullet, \circ, \circ, \bullet, \bullet)$ is equal to 
$\frac{\langle W| DEEDD|V\rangle}{Z_{5}}.$

Uchiyama generalized the Matrix Ansatz to the setting of the two-species exclusion process.
\begin{theorem}\cite{Uchiyama}\label{ansatz2}
Suppose that there are matrices $D$, $E$, $A$, 
and vectors $\langle W|$ and $|V\rangle$ such that the following relations hold:
\begin{itemize}
\item $\langle W| (\alpha E-\gamma D) = \langle W|$
\item $(\beta D-\delta E)|V\rangle = |V\rangle$
\item  $DE-qED = D+E$
\item $DA = qAD+A$
\item $AE = qEA+A$.
\end{itemize}
Then in the two-species ASEP on a lattice of $N$ sites with precisely $r$ 
light particles, the steady state probability of state $(\tau_1,\dots, \tau_n)$ is equal to 
$$\frac{\langle W| \prod_{i=1}^N [\indicator(\tau_i=2)D+\indicator(\tau_i=1)A+\indicator(\tau_i=0)E]|V\rangle}{[y^r]
\langle W| (D+E+yA)^N | V \rangle}.$$
\end{theorem}

For example, the steady state probability of state 
$(2,0, 1,1, 2,2)$ is equal to 
$\frac{\langle W| DEAADD|V\rangle}{Z_{6,2}}.$

\subsection{A solution to the Matrix Ansatz}\label{sec:sol}

Now consider the following tridiagonal matrices,
which were introduced by Uchiyama, Sasamoto and Wadati in 
\cite{USW}.\footnote{Actually we have slightly modified 
the definitions of $d_i^{\sharp}$, $d_i^{\flat}$, $e_i^{\sharp}$, and 
$e_i^{\flat}$ which appeared in \cite{USW} but it is easy to check that 
both the matrices  we are using here and the original matrices of \cite{USW}
satisfy Lemma \ref{thm:USW-solution}.}

\begin{eqnarray*}
\dd=\left[
\begin{array}{cccc}
d_0^\natural 	& d_0^\sharp 	& 0	 	& \cdots\\
d_0^\flat 	& d_1^\natural 	& d_1^\sharp 	& {}\\
0 		& d_1^\flat 	& d_2^\natural 	& \ddots\\
\vdots 		& {} 		& \ddots	& \ddots
\end{array}
\right] ,
&&\qquad 
\ee=\left[
\begin{array}{cccc}
e_0^\natural 	& e_0^\sharp 	& 0 		& \cdots\\
e_0^\flat 	& e_1^\natural 	& e_1^\sharp 	& {}\\
0 		& e_1^\flat 	& e_2^\natural 	& \ddots\\
\vdots 		& {}		& \ddots	& \ddots
\end{array}
\right], \text{ where }
\label{eqn:repde2}
\end{eqnarray*}
\begin{eqnarray*}
d_n^\natural:= d_n^\natural(a,b,c,d) &=&
\frac{q^{n-1}}{(1-q^{2n-2}abcd)(1-q^{2n}abcd)}\nonumber\\
&&\times[
bd(a+c)+(b+d)q-abcd(b+d)q^{n-1}-\{ bd(a+c)+abcd(b+d)\} q^n\nonumber\\
&&-bd(a+c)q^{n+1}+ab^2 cd^2(a+c) q^{2n-1}+abcd(b+d)q^{2n} ] ,\\
e_n^\natural:=e_n^\natural(a,b,c,d) &=&
\frac{q^{n-1}}{(1-q^{2n-2}abcd)(1-q^{2n}abcd)}\nonumber\\
&&\times[
ac(b+d)+(a+c)q-abcd(a+c)q^{n-1}-\{ ac(b+d)+abcd(a+c)\} q^n\nonumber\\
&&-ac(b+d)q^{n+1}+a^2 bc^2 d(b+d) q^{2n-1}+abcd(a+c)q^{2n} ] ,
\label{eqn:elements}
\end{eqnarray*}
\begin{eqnarray*}
&&d_n^\sharp := d_n^\sharp(a,b,c,d) =
1,\qquad \qquad
d_n^\flat:=d_n^\flat(a,b,c,d) =
-\frac{q^nbd}{(1-q^nac)(1-q^nbd)}\mathcal{A}_n \\
&&
e_n^\sharp:=e_n^\sharp(a,b,c,d) =
-q^nac,
\qquad
e_n^\flat:=e_n^\flat(a,b,c,d) =
\frac{1}{(1-q^nac)(1-q^nbd)}\mathcal{A}_n, 
\text{ and }
\end{eqnarray*}
\begin{align*}
\mathcal{A}_n:=&\mathcal{A}_n(a,b,c,d) \\
=&
\frac{(1-q^{n-1}abcd)(1-q^{n+1})
(1-q^nab)(1-q^nac)(1-q^nad)(1-q^nbc)(1-q^nbd)(1-q^ncd)}
{(1-q^{2n-1}abcd)(1-q^{2n}abcd)^2(1-q^{2n+1}abcd)}.
\end{align*}

\begin{remark}
These matrices have the property that the coefficients in the 
$n$th row of $\dd+\ee$ are the coefficients in the three-term
recurrence for the Askey-Wilson polynomials.
\end{remark}

Lemma \ref{thm:USW-solution} below uses the following change of variables
between $a,b,c,d$ and $\alpha, \beta,\gamma, \delta$.
\begin{eqnarray}
a&=&\frac{1-q-\alpha+\gamma+\sqrt{(1-q-\alpha+\gamma)^2+4\alpha\gamma}}{2\alpha} \label{subs3}\\
c&=&\frac{1-q-\alpha+\gamma-\sqrt{(1-q-\alpha+\gamma)^2+4\alpha\gamma}}{2\alpha} \label{subs4}\\
b&=&\frac{1-q-\beta+\delta+\sqrt{(1-q-\beta+\delta)^2+4\beta\delta}}{2\beta} \label{subs5}\\
d&=&\frac{1-q-\beta+\delta-\sqrt{(1-q-\beta+\delta)^2+4\beta\delta}}{2\beta}. \label{subs6}
\end{eqnarray}
Note that this change of variables can be inverted via
\begin{align}
\alpha&=\frac{1-q}{1+ac+a+c},~~~~~ \label{subs1}
&\beta=\frac{1-q}{1+bd+b+d},\\
\gamma&=\frac{-(1-q)ac}{1+ac+a+c},~~~~~\label{subs2}
&\delta=\frac{-(1-q)bd}{1+bd+b+d}.
\end{align}

\begin{lemma}\cite{USW}\label{thm:USW-solution}
Let $D = \frac{1}{1-q}(\indicator+\dd)$ and 
$E = \frac{1}{1-q}(\indicator+\ee)$, where $\indicator$ 
 is the identity matrix.
Also, use the equations \eqref{subs3} through \eqref{subs6}
to express $D$ and $E$ in terms of $\alpha, \beta, \gamma, \delta$, and $q$.
Let $\langle W |= (1,0,0,\dots)$ and 
$|V\rangle = \langle W|^T$.  Then 
$D$, $E$, $\langle W|$, and $|V\rangle$ give a solution to the 
Matrix Ansatz of Theorem \ref{ansatz}.
\end{lemma}

Moreover, Uchiyama \cite{Uchiyama} observed that if 
$A:=DE-ED$, and $D$ and $E$ satisfy $DE-qED = D+E$, then 
it follows that $DA = qAD+A$ and $AE = qEA+A$.  Therefore 
we have the following.

\begin{corollary}\label{cor:ansatz}
Let $D$, $E$, $\langle W|$, and $|V\rangle$ be as in Lemma \ref{thm:USW-solution},
and set $A:=DE-ED$.  Then 
$D$, $E$, $A$, $\langle W|$, and $|V\rangle$ give a solution to the 
Matrix Ansatz of Theorem \ref{ansatz2}.
\end{corollary}

In the rest of the paper, we will use the solution of the matrix ansatz in Corollary \ref{cor:ansatz}, 
which in turn fixes our definition of the following versions of the partition function.
(Note that one gets a definition of the partition function from every solution to the Matrix Ansatz;
using a different solution will simply rescale the partition function.)

\begin{definition}\label{def:ZN}
Let $Z_N(\xi) = \langle W| (\xi D+E)^N | V \rangle$ be the \emph{fugacity partition function} for the ASEP where $D$, $E$, 
$\langle W|$, and $|V\rangle$ are as in Lemma \ref{thm:USW-solution}.
We also let $Z_{N,r}(\xi) = [y^r] \frac{\langle W| (\xi D+E+yA)^N |V \rangle}{\langle W| A^r|V \rangle}$ 
be the \emph{fugacity partition function} for the two-species ASEP.
Finally we set $Z_N=Z_N(1)$ and $Z_{N,r} = Z_{N,r}(1)$ and refer to these quantities as \emph{partition functions}.
\end{definition}

\section{Koornwinder moments and partial Motzkin paths}\label{sec:Motzkin}

We start by reviewing the connection between the partition function $Z_N$ of the ASEP
and moments of Askey-Wilson polynomials.  We will then explain how 
to interpret $Z_N$ as a generating function for weighted Motzkin paths.
Finally we will use the
celebrated lemma of 
Karlin-McGregor-Lindstr\"om-Gessel-Viennot \cite{KM1, KM2, Lindstrom, GV}
on determinants and non-intersecting lattice paths
to express Koornwinder moments in terms of 
\emph{partial} weighted  Motzkin paths.

\subsection{The partition function of the ASEP, Askey-Wilson moments,
and Koornwinder moments}\label{sec:moments2}

Using results from  \cite{USW}, we proved in \cite{CSSW} that there is a close relationship between 
the fugacity partition function $Z_N$ of the ASEP and Askey-Wilson moments $\mu_N$.


\begin{theorem}\cite[Theorem 1.11]{CSSW}\label{thm:ZN-AW}
Recall that $Z_N(\xi) = Z_N(\xi; \alpha, \beta, \gamma, \delta;q) = 
\langle W| (\xi D+E)^N | V \rangle$ be the \emph{fugacity partition function} of the
ASEP, where $D$, $E$, 
$\langle W|$, and $|V\rangle$ are as in Lemma \ref{thm:USW-solution}.
Then 
the $N^{th}$ Askey-Wilson moment    is
equal to
$$
\mu_N(a,b,c,d\vert q)= 
\frac{(1-q)^N}{2^Ni^N} 
Z_N(-1;\alpha,\beta,\gamma,\delta;q),\footnote{The formula we give here
does not include the factor 
$\prod_{j=0}^{N-1}(\alpha\beta-\gamma\delta q^j)$, because we are working 
here with the partition function coming from the $D$ and $E$ from Lemma \ref{thm:USW-solution},
rather than the rescaled version which was used 
in \cite{CSSW}.}
$$
where $i^2=-1$ and
\begin{equation}\label{Subs2}
\alpha=\frac{1-q}{1-ac+ai+ci},~~~~~
\beta=\frac{1-q}{1-bd-bi-di},~~~~~
\gamma=\frac{(1-q)ac}{1-ac+ai+ci},~~~~~
\delta=\frac{(1-q)bd}{1-bd-bi-di}.
\end{equation}
\end{theorem}

Recall from Lemma \ref{lem:Koornwindermoment}
our determinantal expression for Koornwinder moments.
Because of the close relationship between the partition function of the 
ASEP and Askey-Wilson moments 
(see Theorem \ref{thm:ZN-AW}), 
we define another version of Koornwinder moments as follows.

\begin{definition}\label{def:Koornwindermoment}
Given a partition $\lambda = (\lambda_1,\lambda_2,\dots, \lambda_m)$, we define
the \emph{Koornwinder moment} at $q=t$ to be  
\begin{equation}\label{def:moment}
K_{\lambda}(\xi)=K_{\lambda}(\xi; \alpha, \beta, \gamma, \delta; q) = \frac{\det(Z_{\lambda_i+m-i+m-j}(\xi))_{i,j=1}^m}{\det(Z_{2m-i-j}(\xi))_{i,j=1}^m},
\end{equation}
where $Z_N$ is as in Definition \ref{def:ZN}.
\end{definition}

The fact that we refer to the quantities $K_{\lambda}(\xi)$ as \emph{Koornwinder moments} is justified by the 
following result.

\begin{proposition}
Let $\lambda=(\lambda_1,\dots,\lambda_m)$ be a partition.
We have that $$M_{\lambda}(a,b,c,d|q) = \left(\frac{1-q}{2i}\right)^{|\lambda|} K_{\lambda}(-1; \alpha, \beta, \gamma, \delta;q),$$
where $|\lambda| = \sum_i \lambda_i$, and 
 $\alpha, \beta, \gamma, \delta$ are related to $a,b,c,d$ as in \eqref{Subs2}.
\end{proposition}
\begin{proof}
Let $\theta = \frac{1-q}{2i}.$
So $Z_N(-1; \alpha, \beta, \gamma, \delta;q) = \theta^{-N} \mu_N(a,b,c,d|q).$
Then we have that 
\begin{align*}
K_{\lambda}(-1) &= \frac{ \det(Z_{\lambda_i+m-i+m-j}(-1))_{{i,j}=1}^m}{\det(Z_{2m-i-j}(-1))_{{i,j}=1}^m}\\
 &= \frac{ \det(\theta^{-(\lambda_i+m-i+m-j)}\mu_{\lambda_i+m-i+m-j})_{{i,j}=1}^m}{\det(\theta^{-(2m-i-j)}\mu_{2m-i-j})_{{i,j}=1}^m}\\
 &= \frac{ \theta^{-(\lambda_1+ \dots + \lambda_m + m(m-1))}  \det(\mu_{\lambda_i+m-i+m-j})_{{i,j}=1}^m}{\theta^{-m(m-1)}\det(\mu_{2m-i-j})_{{i,j}=1}^m}\\
 &= \theta^{-(\lambda_1 + \dots + \lambda_m)} M_{\lambda}.
\end{align*}

\end{proof}

From now on, when we refer to Koornwinder moments, we will be referring 
to Definition \ref{def:Koornwindermoment}.

Because of its probabilistic interpretation, the quantity
$Z_{N}$ must be positive whenever we specialize the parameters
$\alpha, \beta, \gamma, \delta$, and $q$ to be positive numbers between
$0$ and $1$.
Moreover, we proved in 
\cite{CW-Duke1, CW-Duke2} that (up to a normalizing factor)
$Z_{N}$ is a polynomial 
in $\alpha, \beta, \gamma, \delta, q$ with positive coefficients; 
it can be expressed as a sum over \emph{staircase tableaux}.  
This prompted Rains to conjecture the following.

\begin{conjecture}\label{conj:pos}
The Koornwinder moment 
$K_{\lambda}(\xi)$ is a polynomial in 
$\alpha, \beta, \gamma, \delta, q,\xi$ with positive coefficients (up to a 
normalizing factor).
\end{conjecture}

Since Theorem \ref{thm:main} implies that 
$K_{(N-r,0,0,\dots,0)}(\xi)$ is proportional to the fugacity partition function
$Z_{N,r}(\xi)$ for the two-species ASEP, it follows that 
the Koornwinder moment
$K_{(N-r,0,0,\dots,0)}$ is positive when we specialize the parameters 
$\alpha, \beta, \gamma, \delta, q$ to be positive numbers between
$0$ and $1$.  Moreover, in \cite{CMW}, we will give a tableaux interpretation
of $Z_{N,r}$, which implies that $Z_{N,r}(\xi)$ is 
a polynomial in $\xi, \alpha, \beta, \gamma, \delta$, and $q$,
with positive coefficients.
Finally, see Section \ref{sec:specialize} for some evidence towards Conjecture \ref{conj:pos}
when $q = \xi = 1$.

\subsection{Weighted Motzkin paths and the partition function of the ASEP}

Recall that a \emph{Motzkin path} of \emph{length $N$} is a path in the $xy$ plane 
from $(0,0)$ to $(N,0)$ which consists of steps northeast $(1,1)$, 
east $(1,0)$, and southeast $(1,-1)$, and never dips below the $x$-axis.
One often associates a \emph{weight} to each of the three kinds of 
steps, based on the height at which the step begins.  The weight of a given 
Motzkin path is then the product of the weights of all of its steps, and such 
a path is called a \emph{weighted Motzkin path}.
If $\mathcal{C}=(c_{ij})$ is a tridiagonal matrix rows and columns indexed by the non-negative integers, then 
we can use each nonzero  entry $c_{ij}$ to weight a step in a Motzkin path 
that starts at height $i$ and ends at height $j$.  We call such a Motzkin path a 
\emph{$\mathcal{C}$-Motzkin path}, though we will sometimes drop the $\mathcal{C}$ if it is understood.

The following lemma is obvious.
\begin{lemma}\label{lem:Motzkin1}
Let $\mathcal{C}$ be a tridiagonal matrix as above, 
and let $\langle W |= (1,0,0,\dots)$ and 
$|V\rangle = \langle W|^T$.  Let $\mathcal{Z}_N = \langle W| \mathcal{C}^N |V \rangle$.
Then $\mathcal{Z}_N$ is the generating function for all $\mathcal{C}$-Motzkin paths of length $N$,
i.e. it is the sum of the weights of all $\mathcal{C}$-Motkin paths
from $(0,0)$ to $(N,0)$.
\end{lemma}




We will now give a general result relating a certain ratio of determinants to 
partial Motzkin paths, which we will use to subsequently relate Koornwinder moments to partial Motzkin paths. 
Note that it is well-known that 
there is a link between (partial) Motzkin paths and 
moments of (one-variable) orthogonal polynomials, see for example
\cite[Chapter 1, Proposition 17]{Viennot1} and also \cite{Viennot2}.
However, our treatment here will be self-contained.

We say a \emph{partial Motzkin path} is a path in the $xy$ plane 
from $(0,0)$ to $(N,r)$ which consists of steps northeast $(1,1)$, 
east $(1,0)$, and southeast $(1,-1)$, and never dips below the $x$-axis.
As before, we can use entries $c_{ij}$ of a tridiagonal matrix $\mathcal{C}$ to associate the 
weight $c_{ij}$ to a step which starts at height $i$ and ends at height $j$.
The weight of a  partial
Motzkin  path is then the product of the weights of all of its steps.
We have the following result.

We define the following ratio of determinants: \begin{equation} \label{eq:genmoment} 
\K_{(\lambda_1,\dots,\lambda_n)} = 
\frac{\det \begin{vmatrix}
\Z_{\lambda_1+2n-2} & \Z_{\lambda_1+2n-3} & \dots & \Z_{\lambda_1+n-1}\\
\Z_{\lambda_2+2n-3} & \Z_{\lambda_2+2n-4} & \dots & \Z_{\lambda_2+n-2}\\
\vdots & \vdots & & \vdots \\
\Z_{\lambda_n+n-1} & \Z_{\lambda_n+n-2} & \dots & \Z_{\lambda_n}
\end{vmatrix}}
{\det \begin{vmatrix}
\Z_{2n-2} & \Z_{2n-3} & \dots & \Z_{n-1}\\
\Z_{2n-3} & \Z_{2n-4} & \dots & \Z_{n-2}\\
\vdots & \vdots & & \vdots \\
\Z_{n-1} & \Z_{n-2} & \dots & \Z_0
\end{vmatrix}}.
\end{equation}

\begin{theorem}\label{thm:det-Motzkin}
Let $\mathcal{C}$ be a tridiagonal matrix, and let $\mathcal{C}\Motz(N,r)$ be the generating function for all 
$\mathcal{C}$- partial Motzkin paths which start at $(0,0)$ and end at $(N,r)$.
Let $\langle W |= (1,0,0,\dots)$, 
$|V\rangle = \langle W|^T$,
and $|V^r \rangle = (0,\dots,0,1,0,0,\dots)^T$, where the $1$ is in the $r$th position
(and coordinates are indexed by the non-negative integers).
Let $\mathcal{Z}_N = \langle W| \mathcal{C}^N |V \rangle$
	and $k_r = \prod_{i=0}^{r-1} c_{i,i+1}.$
Then we have that 
\begin{equation}\label{eq:moment}
\frac{1}{k_r} \mathcal{C}\Motz(N,r)
= \frac{1}{k_r} 
{\langle W|  \mathcal{C}^N | V^r\rangle} = 
\K_{(N-r,0,0,\dots,0)},
\end{equation}
where there are precisely $r$ $0$'s in $(N-r,0,0,\dots,0)$.
\end{theorem}

\begin{proof}
The leftmost equality in \eqref{eq:moment} is obvious.  To relate these quantities to the ratio of determinants 
\begin{equation}
\K_{(N-r,0,0,\dots,0)} = 
\frac{\det \begin{vmatrix}
\Z_{N+r} & \Z_{N+r-1} & \dots & \Z_N\\
\Z_{2r-1} & \Z_{2r-2} & \dots & \Z_{r-1}\\
\vdots & \vdots & & \vdots \\
\Z_r & \Z_{r-1} & \dots & \Z_0
\end{vmatrix}}
{\det \begin{vmatrix}
\Z_{2r} & \Z_{2r-1} & \dots & \Z_r\\
\Z_{2r-1} & \Z_{2r-2} & \dots & \Z_{r-1}\\
\vdots & \vdots & & \vdots \\
\Z_r & \Z_{r-1} & \dots & \Z_0
\end{vmatrix}}
\end{equation}
at the right, 
we now use the well-known
Karlin-McGregor-Lindstr\"om-Gessel-Viennot Lemma \cite{KM1, KM2, Lindstrom, GV}
(which we will henceforth refer to as the KMLGV Lemma).  This lemma  
will give a combinatorial
interpretation of both the numerator and denominator.

\begin{figure}[h]
\centering
\includegraphics[height=1.5in]{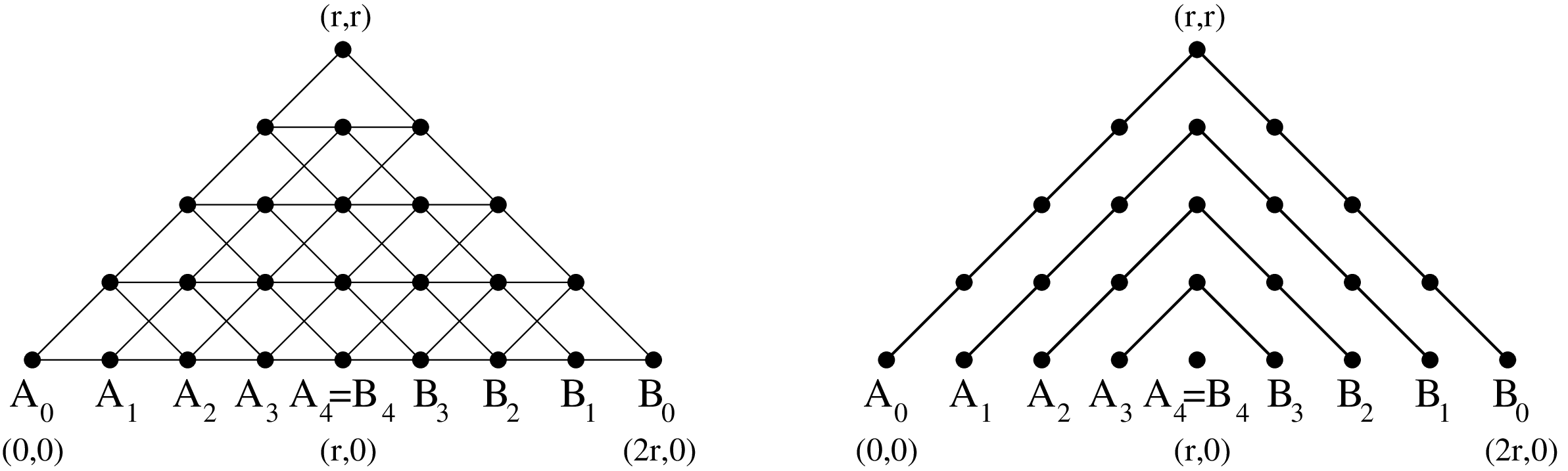}
\caption{An acyclic directed graph (we consider all edges to be 
directed either east, northeast, or southeast) 
corresponding to the denominator of \eqref{eq:moment} when $r=4$, and the unique collection
of pairwise vertex-disjoint paths from $\{A_0,\dots,A_r\}$ to 
$\{B_0,\dots,B_r\}$ in this graph.}
\label{fig:denom}
\end{figure}

Consider the acyclic directed graph shown at the left of 
Figure \ref{fig:denom} (for the case that 
$r=4$).  Although we haven't shown the orientations of the edges,
we consider all edges to be directed either east, northeast, or southeast.  
The points $A_0$, $A_1$,\dots, $A_r=B_r$, 
$B_{r-1}$, $B_{r-2}$, \dots, $B_0$ are at the lattice points 
$(0,0)$, $(1,0)$, \dots, $(r,0)$,$(r+1,0)$,$(r+2,0)$, \dots, $(2r,0)$, respectively.
We weight each edge of the graph which goes from height $i$ to height $j$ by 
the entry $c_{ij}$ of $\mathcal{C}$.
Then we define an $(r+1) \times (r+1)$ \emph{weight matrix} $M^{\denom} = (M^{\denom}_{ij}) $ as follows.
Its rows and columns are indexed by $\{0,1,\dots, r\}$, and 
$M^{\denom}_{ij}$ is defined to be the weights of all paths in the graph 
from $A_i$ to $B_j$.  In other words, 
$M^{\denom}_{ij}$ is the generating function for all 
$\mathcal{C}$-Motzkin paths of length $2r-i-j$.  
Applying Lemma \ref{lem:Motzkin1},
we have that $M^{\denom}_{ij} = \Z_{2r-i-j}$, and hence
$M^{\denom}$ is the matrix appearing in the denominator of \eqref{eq:moment}.

But now by the KMLGV Lemma, the determinant $\det M^{\den}$ is the generating function
for pairwise vertex-disjoint path collections from $\{A_0,\dots, A_r\}$ to 
$\{B_0,\dots, B_r\}$.  It is easy to see that there is only one such path 
collection, which is shown at the right of Figure \ref{fig:denom}. 
Note  that \cite[Theorem 11]{Krattenthaler} can also be used
to give a formula for this Hankel determinant.

\begin{figure}[h]
\centering
\includegraphics[height=1.4in]{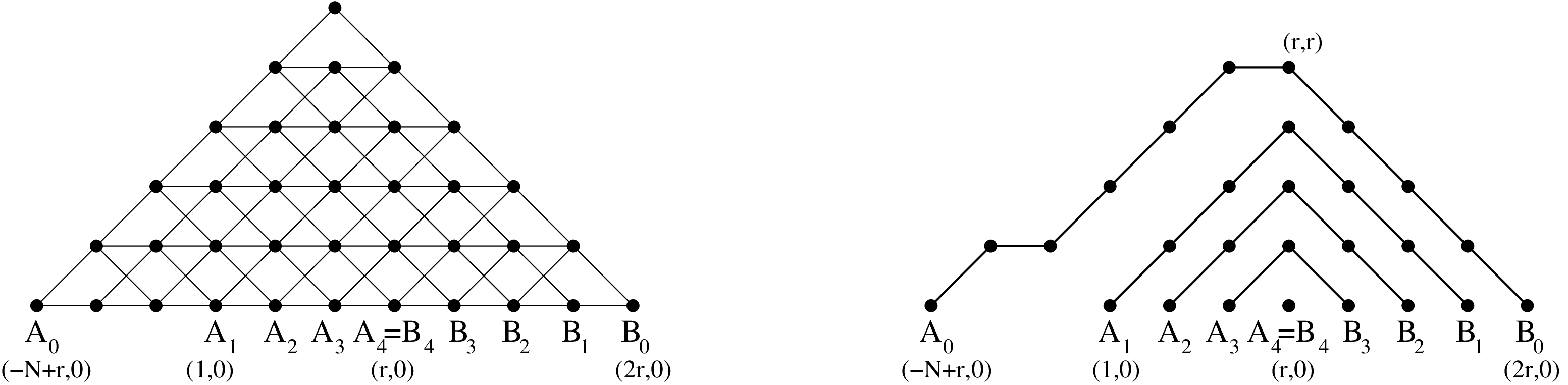}
\caption{An acyclic directed graph (we consider all edges to be 
directed either east, northeast, or southeast) 
corresponding to the numerator of \eqref{eq:moment} when $r=4$.  Now 
there are multiple collections
of pairwise vertex-disjoint paths from $\{A_0,\dots,A_r\}$ to 
$\{B_0,\dots,B_r\}$, but all of them use 
the set of paths from $A_i$ to $B_i$ for $1 \leq i \leq 4$ shown 
at the right of the figure.}
\label{fig:num}
\end{figure}

Now consider the acyclic directed graph shown at the left of 
Figure \ref{fig:num} (for the case that 
$r=4$).  Again
we consider all edges to be directed either east, northeast, or southeast.  
The points $A_0$, $A_1$,\dots, $A_r=B_r$, 
$B_{r-1}$, $B_{r-2}$, \dots, $B_0$ are at the lattice points 
$(-N+r,0)$, $(1,0)$, \dots, $(r,0)$,$(r+1,0)$,$(r+2,0)$, \dots, $(2r,0)$, respectively.
We again weight 
edges of the graph using entries of $\mathcal{C}$.
Then we define the $(r+1) \times (r+1)$ matrix $M^{\num} = (M^{\num}_{ij}) $,
with rows and columns indexed by $\{0,1,\dots, r\}$, as follows.
$M^{\num}_{ij}$ is defined to be the weights of all Motzkin paths  
from $A_i$ to $B_j$.  
Using Lemma \ref{lem:Motzkin1}, it follows that 
$M^{\num}$ is the matrix appearing in the numerator of \eqref{eq:moment}.

By the KMLGV Lemma, the determinant $\det M^{\num}$ is the generating function
for pairwise vertex-disjoint path collections from $\{A_0,\dots, A_r\}$ to 
$\{B_0,\dots, B_r\}$.  It is easy to see that such path collections must 
consist of the paths from $A_i$ to $B_i$ for $1\leq i \leq r$ shown 
at the right of Figure \ref{fig:num}, together with an arbitrary Motzkin path
from $A_0$ to $B_0$ which goes through the point $(r,r)$ and whose last $r$
steps are southeast.

We now have combinatorial interpretations for both $\det M^{\num}$ and $\det M^{\den}$.
Comparing them, we see that the ratio
$\frac{\det M^{\num}}{\det M^{\den}}$ is equal to the generating function
for the partial Motzkin paths from $(-N+r,0)$ to $(r,r)$ divided by the weight of 
the unique partial Motzkin path from $(0,0)$ to $(r,r)$.  
Clearly the generating function for the partial Motzkin paths from $(-N+r,0)$ to 
$(r,r)$ equals the generating function for the partial Motzkin paths from $(0,0)$
to $(N,r)$.  
And the weights of the up
steps in our Motzkin paths are $c_{i,i+1}$, so the weight of 
that unique partial Motzkin path is 
$c_{0,1} c_{1,2}  \dots c_{r-1,r}$.
This completes the proof.
\end{proof}

\begin{corollary}\label{thm:moment-Motzkin}
Now we choose the tridiagonal matrix $C = \xi D+E$,
where $D$ and $E$ are as in Lemma \ref{thm:USW-solution}.  
Let $C\Motz(N,r)$ be the generating function for all 
$C$- partial Motzkin paths which start at $(0,0)$ and end at $(N,r)$.
Then we have that 
\begin{equation*}
K_{(N-r,0,0,\dots,0)} = 
\frac{1}{k_r} 
{\langle W|  C^N | V^r\rangle} = 
\frac{1}{k_r} C\Motz(N,r),
\end{equation*}
where there are precisely $r$ $0$'s in  $(N-r,0,0,\dots,0)$,
	and $k_r = \frac{\prod_{i=0}^{r-1} (\xi-q^i ac)}{(1-q)^r}.$
\end{corollary}

To prove Theorem \ref{thm:main}, we now need to relate
$\langle W| (\xi D+E)^N | V^r\rangle$ to the fugacity partition function for the 2-species ASEP.

\section{A Jacobi-Trudi formula for general Koornwinder moments}\label{sec:JT}


In this section we will prove a general 
Jacobi-Trudi type result in Theorem \ref{thm:JT0}, 
which expresses the quantity $\K_{\lambda}$ (see \eqref{eq:genmoment})
in terms of 
the quantities $\K_{(m,0,\dots,0)}$.
Our proof techniques build on those used 
in the proof of Theorem \ref{thm:moment-Motzkin}.  
We again let $\mathcal{C}=(c_{ij})$ be
a tridiagonal matrix.  Edges of graphs
that we will construct here 
are all directed northeast, east, or southeast, and each edge from 
height $i$ to $j$ is weighted by $c_{ij}$.  

\begin{theorem}\label{thm:JT0}
Let $\lambda = (\lambda_1,\dots, \lambda_n)$ be a partition.  Then 
\begin{equation}\label{eq:JT}
\K_{\lambda} = \det(\K_{(\lambda_i+j-i,0,0,\dots,0)})_{i,j=1}^n,
\end{equation}
where the term on the right-hand-side has precisely $n-j$ $0$'s.
\end{theorem}

Note that Theorem \ref{thm:JT0} immediately implies 
Corollary \ref{thm:JT}, which expresses a general Koornwinder moment
$K_{\lambda}(\xi)$ in terms of the ``complete homogeneous" Koornwinder
moments $K_{(m,0,\dots,0)}(\xi)$.

\begin{corollary}\label{thm:JT}
Let $\lambda = (\lambda_1,\dots, \lambda_n)$ be a partition.  Then 
\begin{equation}
K_{\lambda}(\xi) = \det(K_{(\lambda_i+j-i,0,0,\dots,0)}(\xi))_{i,j=1}^n,
\end{equation}
where the term on the right-hand-side has precisely $n-j$ $0$'s.
\end{corollary}

We now prove Theorem \ref{thm:JT0}.
\begin{proof}
We start by analyzing the left-hand side of \eqref{eq:JT}.
By definition, we have that 
\begin{equation}\label{JT-matrix}
\K_{(\lambda_1,\dots,\lambda_n)} = 
\frac{\det \begin{vmatrix}
\Z_{\lambda_1+2n-2} & \Z_{\lambda_1+2n-3} & \dots & \Z_{\lambda_1+n-1}\\
\Z_{\lambda_2+2n-3} & \Z_{\lambda_2+2n-4} & \dots & \Z_{\lambda_2+n-2}\\
\vdots & \vdots & & \vdots \\
\Z_{\lambda_n+n-1} & \Z_{\lambda_n+n-2} & \dots & \Z_{\lambda_n}
\end{vmatrix}}
{\det \begin{vmatrix}
\Z_{2n-2} & \Z_{2n-3} & \dots & \Z_{n-1}\\
\Z_{2n-3} & \Z_{2n-4} & \dots & \Z_{n-2}\\
\vdots & \vdots & & \vdots \\
\Z_{n-1} & \Z_{n-2} & \dots & \Z_0
\end{vmatrix}}.
\end{equation}

Using the same arguments as in the proof of Theorem \ref{thm:moment-Motzkin}, we interpret
the matrix in the numerator  of \eqref{JT-matrix} as the weight matrix associated to the acyclic directed graph
at the left of Figure \ref{fig:gen-num}, where edges from height $i$ to height $j$
are weighted by $c_{i j}$.
\begin{figure}[h]
\centering
\includegraphics[height=1.8in]{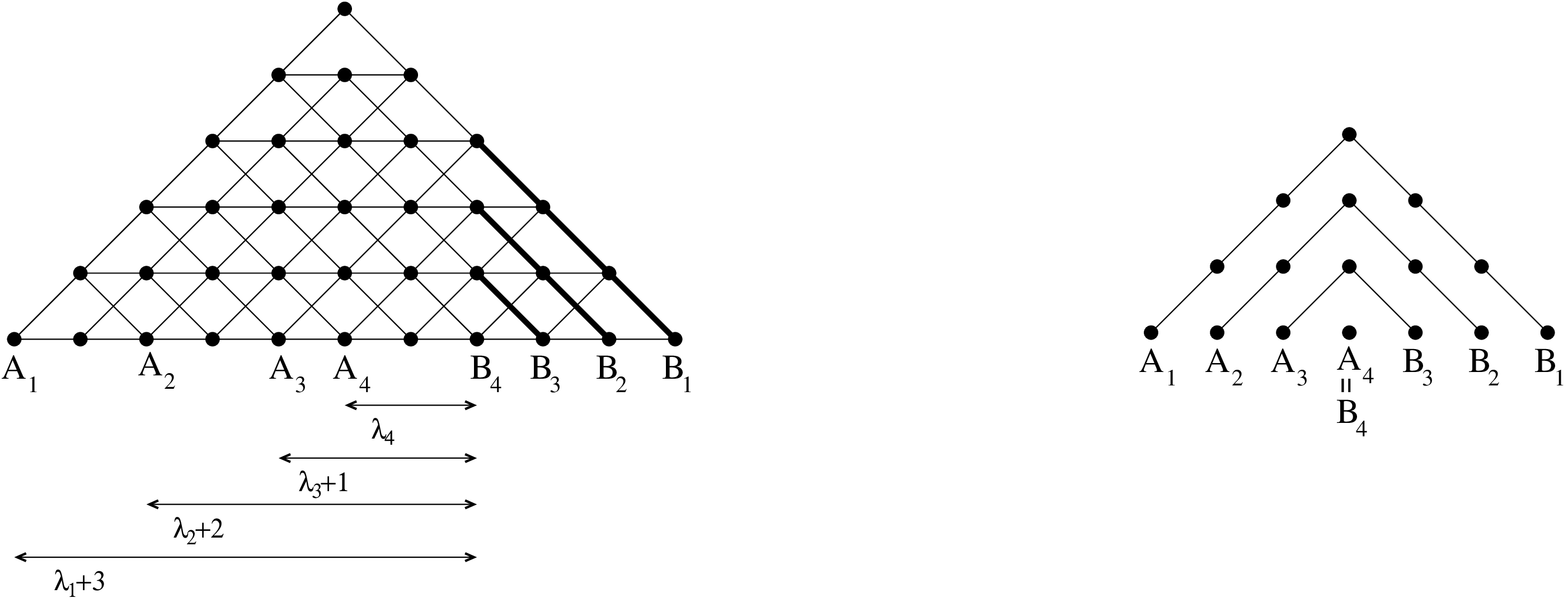}
\caption{An acyclic directed graph (all edges are
directed either east, northeast, or southeast), whose weight matrix is 
given by \eqref{JT-matrix}, when $n=4$.}
\label{fig:gen-num}
\end{figure}
Note that the points $A_1$, \dots, $A_n$, and $B_n$, \dots, $B_1$ all lie
on the $x$-axis, with the $B_i$'s at unit distance apart.  The positions
of the $A_i$'s have been chosen so that the distance between $A_i$ and $B_n$
is $\lambda_i+n-i$.
By Lemma \ref{lem:Motzkin1}, 
the $ij$ entry $\Z_{\lambda_i+2n-i-j}$ of the matrix in the numerator of \eqref{JT-matrix} 
is equal to 
the generating function for Motzkin paths from $A_i$ to $B_j$, hence
this matrix is the weight matrix associated to the graph at the left of Figure \ref{fig:gen-num}.
By the KMLGV Lemma, its determinant 
is the generating function
for pairwise vertex-disjoint path collections from $\{A_1,\dots, A_r\}$ to 
$\{B_1,\dots, B_r\}$.  Note that all such path collections must use the southeast edges
which are shown in bold at the left of Figure \ref{fig:gen-num}.
Moreover, as we showed in the proof of Theorem \ref{thm:moment-Motzkin}, the 
determinant in the denominator of \eqref{JT-matrix} is equal to the weight of the 
path collection shown at the right of Figure \ref{fig:gen-num}.
Therefore the ratio of the determinants is equal to the generating function for 
pairwise-disjoint non-intersecting paths from 
$A_1,\dots, A_n$ to the lattice points $B_1,\dots,B_4$ in Figure 
\ref{fig:JT2}, divided by 
the product of the weights of all the up steps in the path collection shown at the 
right of Figure \ref{fig:gen-num}.  That product is equal to 
$k_1 k_2 \dots k_{n-1}$, where $k_r = \prod_{i=0}^{r-1} c_{i,i+1}.$

\begin{figure}[h]
\centering
\includegraphics[height=1.8in]{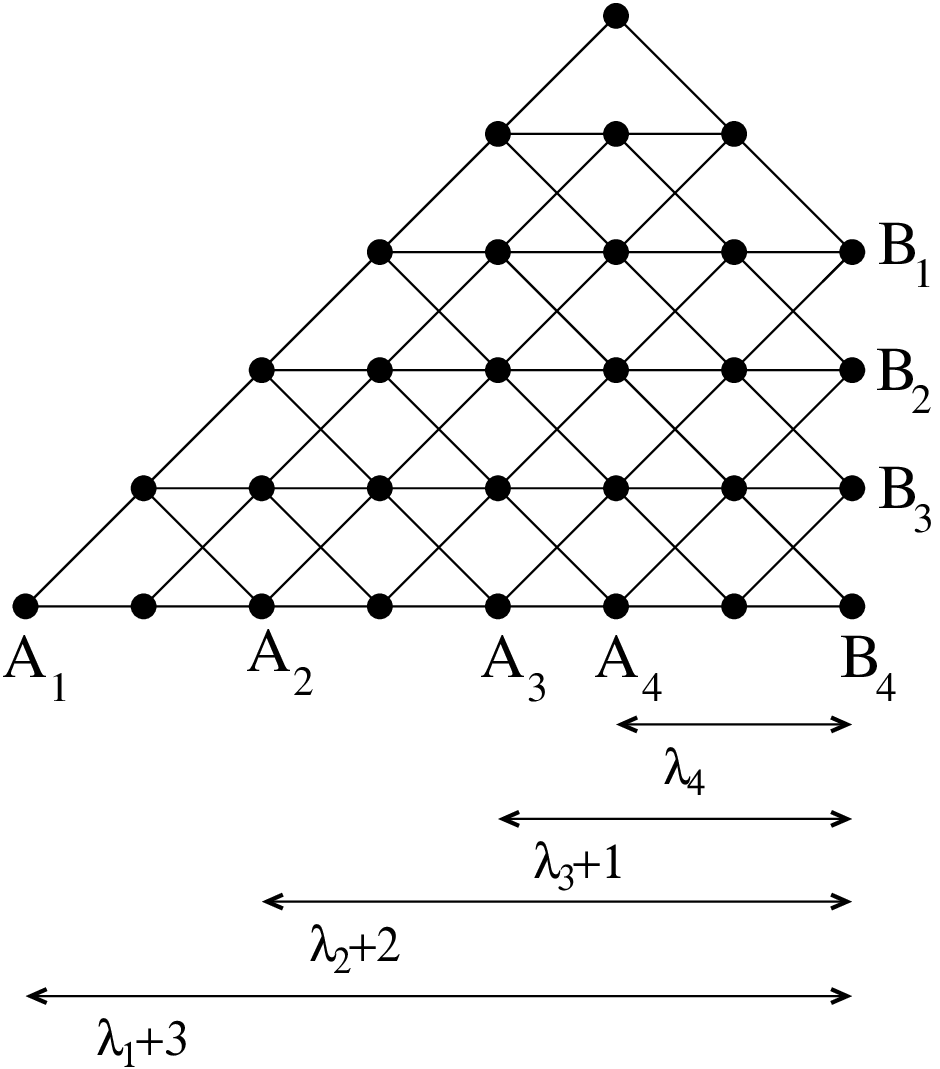}
\caption{An acyclic directed graph, shown for $n=4$.  Edges are oriented
northeast, east, or southeast.}
\label{fig:JT2}
\end{figure}

We now consider the right-hand side of \eqref{eq:JT}.
We have that 
\begin{equation}\label{eq:det}
\det(\K_{(\lambda_i+j-i,0,0,\dots,0)})_{i,j=1}^n = 
\det \begin{vmatrix}
\K_{(\lambda_1,0,\dots,0)} & \K_{(\lambda_1+1,0,\dots,0)} & \dots & \K_{(\lambda_1+n-1)}\\
\K_{(\lambda_2-1,0,\dots,0)} & \K_{(\lambda_2,0,\dots,0)} & \dots & \K_{(\lambda_2+n-2)}\\
\vdots & \vdots & & \vdots \\
\K_{(\lambda_n-n+1,0,\dots,0)} & \K_{(\lambda_n-n+2,0,\dots,0)} & \dots & \K_{(\lambda_n)}
\end{vmatrix},
\end{equation}
where the partitions in column $j$ of the matrix contain precisely $n-j$ $0$'s.  
Using Corollary \ref{thm:moment-Motzkin}, we have that \eqref{eq:det} is equal to 
the determinant of 
\begin{equation*}
\begin{vmatrix}
\frac{1}{k_{n-1}}  \C Motz(\lambda_1+n-1,n-1) & 
\frac{1}{k_{n-2}}  \C Motz(\lambda_1+n-1,n-2) & \dots &
                   \C Motz(\lambda_1+n-1,0)\\
\frac{1}{k_{n-1}}  \C Motz(\lambda_2+n-2,n-1) & 
\frac{1}{k_{n-2}}  \C Motz(\lambda_2+n-2,n-2) & \dots &
                   \C Motz(\lambda_2+n-2,0)\\
\vdots & \vdots & & \vdots \\
\frac{1}{k_{n-1}}  \C Motz(\lambda_n,n-1) & 
\frac{1}{k_{n-2}}  \C Motz(\lambda_n,n-2) & \dots &
                   \C Motz(\lambda_n,0)\\
\end{vmatrix}, 
\end{equation*}
which is equal to 
\begin{equation}\label{eq:wtmatrix}
\det \begin{vmatrix}
 \C Motz(\lambda_1+n-1,n-1) & 
  \C Motz(\lambda_1+n-1,n-2) & \dots &
                  \C Motz(\lambda_1+n-1,0)\\
 \C Motz(\lambda_2+n-2,n-1) & 
  \C Motz(\lambda_2+n-2,n-2) & \dots &
                  \C Motz(\lambda_2+n-2,0)\\
\vdots & \vdots & & \vdots \\
 \C Motz(\lambda_n,n-1) & 
 \C Motz(\lambda_n,n-2) & \dots &
                   \C Motz(\lambda_n,0)\\
\end{vmatrix}
\prod_{\ell=1}^{n-1} \frac{1}{k_{\ell}}. 
\end{equation}
But now it's clear that the matrix in \eqref{eq:wtmatrix} is the weight matrix
for the directed graph shown in Figure 
\ref{fig:JT2}, so we are done.
\end{proof}

\section{From partial Motzkin paths to the  two-species ASEP}\label{sec:M-2}

Our goal in this section is to prove Theorem \ref{thm:main2}, which expresses
the fugacity partition function $Z_{N,r}(\xi)$
in terms of partial Motzkin paths.  We also use this
theorem to give an integral representation and an 
explicit closed formula for $Z_{N,r}(\xi)$ (the latter was communicated to us
by Dennis Stanton).

\begin{theorem}\label{thm:main2}
Define $D$, $E$, $\langle W|$, and $|V\rangle$ as in Lemma \ref{thm:USW-solution},
and set $A = DE-ED$.  
Using the change of variables from \eqref{subs1} and \eqref{subs2}, 
define 
\begin{equation}\label{eq:rho}
\rho_r = \frac{(1-q)^r}{\prod_{i=0}^{r-1} (\xi-q^i ac)} = 
\frac{1}{k_r} = \frac{\alpha^r (1-q)^r}{\prod_{i=0}^{r-1} (\alpha \xi+q^i \gamma)}.
\end{equation}
  Then we have that 
\begin{equation}\label{eq:partition}
\langle W| (\xi D+E)^N | V^r\rangle \cdot \rho_r = 
[y^r] \frac{\langle W| (\xi D+E+yA)^N |V \rangle}{\langle W| A^r|V \rangle}
= Z_{N,r}(\xi).
\end{equation}
Note that 
$\langle W| (\xi D+E)^N | V^r\rangle$ is the generating function
for $\C$-partial Motzkin paths, with $\C = \xi D+E$, so this
gives a combinatorial interpretation for the fugacity partition function
$Z_{N,r}(\xi)$ of  the 
two-species ASEP.
\end{theorem}

Note that our main theorem, Theorem \ref{thm:main},
is an immediate consequence of Corollary \ref{thm:moment-Motzkin}
and Theorem \ref{thm:main2}.

Before proving Theorem \ref{thm:main2}, we 
observe that it directly implies an integral representation for 
the fugacity partition function.

\begin{corollary}\label{cor:integral}
Using the change of variables from \eqref{subs1} and \eqref{subs2}, 
we have the following integral representation
for the fugacity partition function:
$$
Z_{N,r}(\xi)=B_{N,r} \int (\sqrt{\xi}+1/\sqrt{\xi}+x)^N w(x,a',b',c',d';q)P_r(x,a',b',c',d';q)\frac{dx}{4i\pi x}
$$
where $B_{N,r}$ is a simple constant, 
$a' =  a \sqrt{\xi}, \ b' = b/\sqrt{\xi}, \ c' =  c\sqrt{\xi} ,  
\ d' =  d/\sqrt{\xi}$, 
$w(x,b,c,d;q)$ is the Askey-Wilson density
and $P_r(x,a,b,c,d;q)$ is the $r^{th}$ Askey-Wilson polynomial.
See Section \ref{sec:Koornwinder} for the definitions of the Askey-Wilson
density and polynomials.
\end{corollary}

\begin{proof}
Let $\C$ be a tridiagonal matrix, whose rows encode 
the recurrence relations defining a family
of orthogonal polynomials $\{P_m(x)\}$.  Let $f$ be the 
linear functional expressing the orthogonality
relation, i.e. $f(P_m(x) P_{\ell}(x)) = 0$ unless $m=\ell$.
By \cite[Proposition 17, Chapter 1]{Viennot2}, 
the quantity $f(x^N P_r(x))$ is proportional to 
 the  sum of the weights of all 
$\C$-partial Motzkin paths of length $N$ which 
start at height $0$ and end at 
height $r$.

Now let $\C = \xi D+E$ be the matrix from Theorem \ref{thm:main2};
we have that $\C = 
\frac{1}{1-q} \sqrt{\xi}\left( \sqrt{\xi} \dd + \frac{1}{\sqrt{\xi}} \ee +
(\sqrt{\xi} + \frac{1}{\sqrt{\xi}}) \indicator \right).$
Using the arguments of \cite[Proposition 2.7]{CSSW}
together with \cite[Proposition 17, Chapter 1]{Viennot2} 
and Theorem \ref{thm:main2} yields the result.
\end{proof}

From Corollary \ref{cor:integral},
one can compute asymptotics, including 
the particle current and the particle densities.  These quantities were originally
computed by Uchiyama \cite{Uchiyama}, and were later recomputed by 
Cantini
\cite{Cantini} 
when $r=N/k$ and $N\rightarrow\infty$.
As expected, the phase diagram has three phases: low density, high density and maximal current.

From Corollary \ref{cor:integral},
it is also possible to generalize \cite[Theorem 1.13]{CSSW} and get a closed 
formula for $Z_{N,r}(\xi)$. This was communicated to us by Dennis Stanton
\cite{Dennis}.  To state this formula, 
let $(a;q)_n=\prod_{i=0}^{n-1}(1-aq^i)$, $(a_1,\ldots ,a_k;q)_n=\prod_{i=1}^k (a_i;q)_n$
and
$$
\left[\begin{array}{l} k\\ r\end{array}\right]=\frac{(q;q)_n}{(q;q)_k(q;q)_{n-k}}.
$$
Then let
$$
F_{k,r}=(-a/\sqrt{\xi})^rq^{r\choose 2}\left[\begin{array}{l} k\\ r\end{array}\right](ab,ac/\xi,ad,bc,bd\xi,cd,q,abcdq^{r-1};q)_r
\frac{(abq^r,acq^r/\xi,adq^r;q)_{k-r}}{(abcdq;q)_{k+r}}.
$$
and
$$
G_{N,r}(\xi)=\sum_{k=r}^{N}\sum_{j=0}^k \frac{F_{k,r} q^{k-j^2}a^{-2j}\left(\xi+1+aq^j+\frac{\xi}{aq^j}\right)^N}
{(q,\xi q^{1-2j}/a^2;q)_j(q;q^{2j+1}a^2/\xi,q;q)_{k-j}}.
$$

\begin{theorem}\cite{Dennis}
The fugacity partition function is equal to
\begin{equation}
Z_{N,r}(\xi)=\frac{\prod_{i={2r}}^{N+r-1}(\alpha\beta-\gamma\delta q^{i})}{(1-q)^{N-r}}\frac{G_{N,r}(\xi)}{G_{r,r}(\xi)}.
\end{equation}
using the change of variables from 
\eqref{subs3}, \eqref{subs4}, \eqref{subs5}, \eqref{subs6}.
\end{theorem}
\begin{proof}
The case $r=0$ is \cite[Theorem 1.13]{CSSW},
 and the proof of this theorem is analogous to the proof
there.
\end{proof}

Now we return to the proof of Theorem \ref{thm:main2}.  We will
actually prove a refinement of this theorem; to state it,
we need to introduce some notation.
Given a word $X$ in $\{D,E\}^N$, let 
$S_r(X)$ be the set of words which can be obtained  from $X$ by replacing
precisely $r$ letters in $X$ by an $A$.
Given a word $Z\in S_r(X)$, let 
$D(Z)$ (respectively, $E(Z)$) 
be the set of positions of letters that were $D$ (respectively, $E$) in $X$
and became $A$ in $Z$.  
Let $$\inv_E(Z) = \sum_{j\in E(Z)} |\{i\in D(Z) \cup E(Z) \ \vert \ i < j\}|.$$

\begin{theorem}\label{thm:refinement}
Let $\tilde \rho_r = \alpha^r (1-q)^r$.
For any word $X$ in 
$\{D,E\}^N$, we have that 
\begin{equation}\label{eq:refinement}
 \langle W| X |V^r \rangle \cdot \tilde \rho_r = 
\sum_{Z\in S_r(X)} q^{\inv_E(Z)} \alpha^{|D(Z)|} \gamma^{|E(Z)|} 
\frac{\langle W|Z|V\rangle}{\langle W|A^r |V\rangle}.
\end{equation}
\end{theorem}

We next explain how Theorem \ref{thm:refinement} refines Theorem \ref{thm:main2}.
\begin{lemma}\label{lem:implies}
Theorem \ref{thm:refinement} implies Theorem \ref{thm:main2}.
\end{lemma}

\begin{proof}
Define $|X|_D$ to be the number of $D$'s that occur in a word $X$.
We take \eqref{eq:refinement} and sum it over all words $X$ in 
$\{D,E\}^N$, keeping track of the number of $D$'s in $X$, and obtaining
\begin{equation}\label{eq:refinement2}
\sum_{X\in \{D,E\}^N} 
\xi^{|X|_D} \langle W| X |V^r \rangle \cdot \tilde \rho_r = 
\sum_{X\in \{D,E\}^N} 
\sum_{Z\in S_r(X)} \xi^{|X|_D} q^{\inv_E(Z)} \alpha^{|D(Z)|} \gamma^{|E(Z)|} 
\frac{\langle W|Z|V\rangle}{\langle W|A^r |V\rangle}.
\end{equation}

On the left-hand side of \eqref{eq:refinement2}, we obtain 
$\langle W| (\xi D+E)^N | V^r \rangle \cdot \tilde \rho_r = 
\langle W| (\xi D+E)^N | V^r \rangle \cdot \rho_r   \cdot \prod_{i=0}^{r-1} (\alpha \xi + q^i \gamma).$

Now we will analyze the right-hand side of \eqref{eq:refinement2}.
Let us fix a word $Z \in \{D, E, A\}^N$ which contains precisely $r$ 
$A$'s, and determine the coefficient of 
$\frac{\langle W|Z|V\rangle}{\langle W|A^r |V\rangle}$. 
Note that $|X|_D = |Z|_D + D(Z)$.  Also,
$\frac{\langle W|Z|V\rangle}{\langle W|A^r |V\rangle}$
will appear 
$2^r$ times on the right-hand side -- 
this is the number of ways that each of those $A$'s in $Z$ could have 
come from a $D$ or an $E$ in a word $X\in \{D,E\}^N$.
Now note that the statistics 
$E(Z), |D(Z)|, |E(Z)|$ only depend on the substring of 
$r$ $A$'s in $Z$, and which letter each of those $A$'s originally corresponded to in $X$.
Let us consider each of the $A$'s in $Z$ from left to right.  
If the first $A$ came from a $D$ (respectively, an $E$), then 
we pick up a factor of $\alpha \xi $ (respectively, $\gamma$) 
in $\xi^{|X|_D} q^{\inv_E(Z)} \alpha^{|D(Z)|} \gamma^{|E(Z)|}$.
If the second $A$ came from a $D$ (respectively, an $E$), then 
we pick up a factor of $\alpha \xi $ (respectively, $q \gamma$) 
in $\xi^{|X|_D} q^{\inv_E(Z)} \alpha^{|D(Z)|} \gamma^{|E(Z)|}$ -- the $q$ comes from 
the inversion between the first and second $A$.  
And in general, if the 
$i$th $A$ came from a $D$ (respectively, an $E$), then 
we pick up a factor of $\alpha \xi$ (respectively, $q^{i-1} \gamma$).  
Therefore the coefficient of 
$\frac{\langle W|Z|V\rangle}{\langle W|A^r |V\rangle}$ on the 
right-hand side of \eqref{eq:refinement2} is 
$(\alpha \xi+\gamma)(\alpha\xi+q\gamma) \dots (\alpha\xi+q^{r-1} \gamma)$.
It follows that the right-hand side equals 
$[y^r]\frac{\langle W|(D\xi +E+yA)^N|V \rangle}{\langle W| A^r| V\rangle} \cdot
\prod_{i=0}^{r-1} (\alpha\xi+q^{i} \gamma)$.

Comparing our expressions for the left-hand side and right-hand side
yields \eqref{eq:partition}.
\end{proof}

\begin{theorem}\label{prop:suffices}
To prove Theorem \ref{thm:refinement}, it suffices to prove Theorem
\ref{thm:refinement} in the case that $X = D^N$.  In other words,
it suffices to prove that 
\begin{equation}\label{eq:suffices}
\langle W|D^N|V^r \rangle \cdot \tilde \rho_r = 
\alpha^r \sum_{Z\in S_r(D^N)} \frac{\langle W|Z|V \rangle}{\langle W|A^r |V\rangle}
= [y^r] \frac{\langle W|(D+y\alpha A)^N|V \rangle}{\langle W|A^r|V \rangle}.
\end{equation}
\end{theorem}
Note that the second equality of \eqref{eq:suffices} is obvious.
Theorem \ref{prop:suffices} is a consequence of the following two lemmas.

\begin{lemma}\label{lem:1}
If Theorem \ref{thm:refinement} is true for all words $X$ of the form
$E^{\ell} D^m$, then it is true for all words $X$ in the letters $D$ and $E$.
\end{lemma}

\begin{proof}
Let $U$ be an arbitrary word in $D$ and $E$. 
By Corollary \ref{cor:ansatz},  the matrices $D$, $E$, and $A$ satisfy the relations 
of Theorem \ref{ansatz2}.  In particular we have that 
$DE = qED+D+E$.  Note that by repeatedly applying the relation 
$DE=qED+D+E$ (to replace instances of $DE$ by $qED+D+E$), 
we can put $U$ into ``normal form" -- that is, we can 
write $U$ as a linear combination of words of the form $E^{\ell} D^m$.
We will prove Lemma \ref{lem:1} by induction on the number of times one must
apply the relation $DE=qED+D+E$ to put $U$ into normal form.  The base case 
is the words which already have the form $E^{\ell} D^m$.

For the inductive step, let us write 
$U = XDEY$, where $X$ and $Y$ are words in $D$ and $E$.
We have that $XDEY = qXEDY+XDY+XEY$.
By the induction hypothesis, Theorem \ref{thm:refinement} is true for 
each of the words $XEDY$, $XDY$, and $XEY$.  We want to show that \eqref{eq:refinement}
also holds for $XDEY$.

Let us first analyze the right-hand side of \eqref{eq:refinement} when applied to 
the word $U = XDEY$.  We will write each word $Z\in S_r(XDEY)$ as 
$Z_X DE Z_Y$ or 
$Z_X AE Z_Y$ or 
$Z_X DA Z_Y$ or 
$Z_X AA Z_Y$, 
where $Z_X$ and $Z_Y$ have been obtained from $X$ and $Y$, respectively, by replacing
some of the letters by $A$'s.  
We will also write $Z_X Z_Y$ to denote the word obtained from one of the four 
words above by deleting the two letters between $Z_X$ and $Z_Y$.

We have that 
$\sum_{Z\in S_r(U)} 
q^{\inv_E(Z)} \alpha^{|D(Z)|} \gamma^{|E(Z)|} 
\frac{\langle W|Z|V\rangle}{\langle W|A^r |V\rangle}$ is equal to 
\begin{align*}
& \sum_{Z=Z_XDEZ_Y \in S_r(U)} 
q^{\inv_E(Z_X Z_Y)} \alpha^{|D(Z_X Z_Y)|} \gamma^{|E(Z_X Z_Y)|} 
\frac{\langle W|Z|V\rangle}{\langle W|A^r |V\rangle}\ \\
+& \sum_{Z=Z_XAEZ_Y \in S_r(U)} 
q^{|E(Z_Y)|+\inv_E(Z_X Z_Y)} \alpha^{|D(Z_X Z_Y)|+1} \gamma^{|E(Z_X Z_Y)|} 
\frac{\langle W|Z|V\rangle}{\langle W|A^r |V\rangle}\ \\
+& \sum_{Z=Z_XDAZ_Y \in S_r(U)} 
q^{|E(Z_Y)|+|D(Z_X)|+|E(Z_X)|+\inv_E(Z_X Z_Y)} \alpha^{|D(Z_X Z_Y)|} \gamma^{|E(Z_X Z_Y)|+1} 
\frac{\langle W|Z|V\rangle}{\langle W|A^r |V\rangle}\ \\
+& \sum_{Z=Z_XAAZ_Y \in S_r(U)} 
q^{2|E(Z_Y)|+|D(Z_X)|+|E(Z_X)|+1+\inv_E(Z_X Z_Y)} \alpha^{|D(Z_X Z_Y)|+1} \gamma^{|E(Z_X Z_Y)|+1} 
\frac{\langle W|Z|V\rangle}{\langle W|A^r |V\rangle}. \\
\end{align*}
Note that in the first three terms above, where $Z$ has the form 
$Z_X DEZ_Y$, $Z_X AE Z_Y$, and $Z_X DA Z_Y$, we can apply the relations 
$DE=qED+D+E$, 
$AE = qEA+A$, and
$DA = qAD+A$, 
to rewrite $\langle W|Z|V\rangle$ as 
$\langle W| Z_X (qED+D+E) Z_Y | V \rangle$,
$\langle W| Z_X (qEA+A) Z_Y | V \rangle$, and
$\langle W| Z_X (qAD+A) Z_Y | V \rangle$, respectively.

Now let us analyze the left-hand side of 
\eqref{eq:refinement} when applied to 
the word $U = XDEY$.  
Applying the re-writing rule $DE=qED+D+E$ and using the induction hypothesis, 
we have that 
$\langle W| X(DE)Y |V^r \rangle \cdot \tilde \rho_r $ is equal to  
\begin{align*}
&(q\langle W| X(ED)Y |V^r \rangle   +
\langle W| XDY |V^r \rangle   +
\langle W| XEY |V^r \rangle) \cdot \tilde \rho_r    \\
=&\sum_{Z\in S_r(XEDY)} 
q^{\inv_E(Z)+1} \alpha^{|D(Z)|} \gamma^{|E(Z)|} 
\frac{\langle W|Z|V\rangle}{\langle W|A^r |V\rangle}\ + 
\sum_{Z\in S_r(XDY)} 
q^{\inv_E(Z)} \alpha^{|D(Z)|} \gamma^{|E(Z)|} 
\frac{\langle W|Z|V\rangle}{\langle W|A^r |V\rangle}\\
&+ \sum_{Z\in S_r(XEY)} 
q^{\inv_E(Z)} \alpha^{|D(Z)|} \gamma^{|E(Z)|} 
\frac{\langle W|Z|V\rangle}{\langle W|A^r |V\rangle}  
\end{align*}
\begin{align*}
=& \sum_{Z=Z_XEDZ_Y \in S_r(XEDY)} 
q^{\inv_E(Z_X Z_Y)+1} \alpha^{|D(Z_X Z_Y)|} \gamma^{|E(Z_X Z_Y)|} 
\frac{\langle W|Z|V\rangle}{\langle W|A^r |V\rangle}\ \\
&+ \sum_{Z=Z_XEAZ_Y \in S_r(XEDY)} 
q^{|E(Z_Y)|+1+\inv_E(Z_X Z_Y)} \alpha^{|D(Z_X Z_Y)|+1} \gamma^{|E(Z_X Z_Y)|} 
\frac{\langle W|Z|V\rangle}{\langle W|A^r |V\rangle}\ \\
&+ \sum_{Z=Z_XADZ_Y \in S_r(XEDY)} 
q^{|D(Z_X)|+|E(Z_X)|+|E(Z_Y)|+1+\inv_E(Z_X Z_Y)} \alpha^{|D(Z_X Z_Y)|} \gamma^{|E(Z_X Z_Y)|+1} 
\frac{\langle W|Z|V\rangle}{\langle W|A^r |V\rangle}\ \\
&+ \sum_{Z=Z_XAAZ_Y \in S_r(XEDY)} 
q^{|D(Z_X)|+|E(Z_X)|+2|E(Z_Y)|+1+\inv_E(Z_X Z_Y)} \alpha^{|D(Z_X Z_Y)|+1} \gamma^{|E(Z_X Z_Y)|+1} 
\frac{\langle W|Z|V\rangle}{\langle W|A^r |V\rangle}\ \\
&+ \sum_{Z=Z_XDZ_Y \in S_r(XDY)} 
q^{\inv_E(Z_X Z_Y)} \alpha^{|D(Z_X Z_Y)|} \gamma^{|E(Z_X Z_Y)|} 
\frac{\langle W|Z|V\rangle}{\langle W|A^r |V\rangle}\ \\
&+ \sum_{Z=Z_XAZ_Y \in S_r(XDY)} 
q^{|E(Z_Y)|+\inv_E(Z_X Z_Y)} \alpha^{|D(Z_X Z_Y)|+1} \gamma^{|E(Z_X Z_Y)|} 
\frac{\langle W|Z|V\rangle}{\langle W|A^r |V\rangle}\ \\
&+ \sum_{Z=Z_XEZ_Y \in S_r(XEY)} 
q^{\inv_E(Z_X Z_Y)} \alpha^{|D(Z_X Z_Y)|} \gamma^{|E(Z_X Z_Y)|} 
\frac{\langle W|Z|V\rangle}{\langle W|A^r |V\rangle}\ \\
&+ \sum_{Z=Z_XAZ_Y \in S_r(XEY)} 
q^{|D(Z_X)|+|E(Z_X)|+|E(Z_Y)|+\inv_E(Z_X Z_Y)} \alpha^{|D(Z_X Z_Y)|} \gamma^{|E(Z_X Z_Y)|+1} 
\frac{\langle W|Z|V\rangle}{\langle W|A^r |V\rangle}.
\end{align*}

Now if we compare our expressions for the left-hand side and right-hand side of 
\eqref{eq:refinement}, we see that they are equal.  This completes the proof. 
\end{proof}

\begin{lemma}\label{lem:2}
If Theorem \ref{thm:refinement} is true for all words $X$ of the form
$D^N$, then it is true for all words $X$ of the form $E^{\ell} D^m$.
\end{lemma}

\begin{proof}
We will prove that Theorem \ref{thm:refinement} is true for any word $U$
containing exactly $\ell$ $E$'s by induction on $\ell$ and on $n$,
the number of applications of $DE=qED+D+E$ which are necessary to put the 
word in normal form.

If $U$ has the form $X (DE) Y$, then the proof of Lemma \ref{lem:1} shows
that the theorem holds for $U$, by rewriting 
$U = X(DE)Y = qX(ED)Y + XDY+XEY$.  (Note that the process of rewriting does
not increase the number of $E$'s in the word.)

Otherwise we can assume that $U$ has the form $U = E^{\ell} D^m$.
Let us first analyze the left-hand side of \eqref{eq:refinement} for the word $U$.
By Corollary \ref{cor:ansatz} and 
Theorem \ref{ansatz2}, we have the relation
$\langle W| E = \frac{1}{\alpha} \langle W| + \frac{\gamma}{\alpha} \langle W|D$.

Using the induction hypothesis,  we have that  $\langle W| E^{\ell} D^m |V^r \rangle \cdot \tilde \rho_r$ is equal to 
\begin{align*}
(\frac{1}{\alpha} &\langle W| E^{\ell-1} D^m |V^r \rangle  +
\frac{\gamma}{\alpha} \langle W| D E^{\ell-1} D^m |V^r \rangle) \cdot \tilde \rho_r\\
=& \sum_{Z \in S_r(E^{\ell-1}D^m)} q^{\inv_E(Z)} \alpha^{|D(Z)|-1} \gamma^{|E(Z)|} 
\frac{\langle W|Z|V\rangle}{\langle W|A^r|V\rangle} \\
&+ 
\sum_{Z=D Z_X \in S_r(D E^{\ell-1}D^m)} q^{\inv_E(Z_X)} \alpha^{|D(Z_X)|-1} \gamma^{|E(Z_X)|+1} 
\frac{\langle W|Z|V\rangle}{\langle W|A^r|V\rangle} \\
&+ 
\sum_{Z=A Z_X \in S_r(D E^{\ell-1}D^m)} q^{|E(Z_X)|+\inv_E(Z)} 
\alpha^{|D(Z_X)|} \gamma^{|E(Z_X)|+1} 
\frac{\langle W|Z|V\rangle}{\langle W|A^r|V\rangle}. 
\end{align*}

Analyzing the right-hand side of \eqref{eq:refinement} for the word $U$,
and again using the relation 
$\langle W| E = \frac{1}{\alpha} \langle W| + \frac{\gamma}{\alpha} \langle W|D$,
we have that 
\begin{align*}
& \sum_{Z=EZ_X \in S_r(E^{\ell}D^m)} q^{\inv_E(Z_X)} \alpha^{|D(Z_X)|} \gamma^{|E(Z_X)|} 
\frac{\langle W|Z|V\rangle}{\langle W|A^r|V\rangle} \\
&+
\sum_{Z=AZ_X \in S_r(E^{\ell}D^m)} q^{|E(Z_X)|+\inv_E(Z_X)} \alpha^{|D(Z_X)|} \gamma^{|E(Z_X)|+1} 
\frac{\langle W|Z|V\rangle}{\langle W|A^r|V\rangle} \\
=& \sum_{Z=EZ_X \in S_r(E^{\ell}D^m)} q^{\inv_E(Z_X)} \alpha^{|D(Z_X)|} \gamma^{|E(Z_X)|} 
\frac{\frac{1}{\alpha} \langle W| Z_X|V\rangle+ \frac{\gamma}{\alpha} \langle W|DZ_X|V \rangle}{\langle W|A^r|V\rangle} \\
&+
\sum_{Z=AZ_X \in S_r(E^{\ell}D^m)} q^{|E(Z_X)|+\inv_E(Z_X)} \alpha^{|D(Z_X)|} \gamma^{|E(Z_X)|+1} 
\frac{\langle W|Z|V\rangle}{\langle W|A^r|V\rangle}.
\end{align*}

Comparing our expressions for the left-hand side and right-hand side of 
\eqref{eq:refinement}, we see that they are equal.  This completes the proof.
\end{proof}

\begin{theorem}\label{thm:finalcheck}
The equation \eqref{eq:suffices} holds.
More specifically, \begin{equation*}
\langle W|D^N|V^r \rangle \cdot \tilde \rho_r = 
[y^r] \frac{\langle W|(D+y\alpha A)^N|V \rangle}{\langle W|A^r|V \rangle}.
\end{equation*}
\end{theorem}

Section \ref{sec:thm} will be devoted to the proof of 
Theorem \ref{thm:finalcheck}.  Note that once we have proved
Theorem \ref{thm:finalcheck}, 
this theorem together with  Theorem \ref{prop:suffices} 
will imply 
Theorem \ref{thm:refinement}.  And then
Theorem \ref{thm:refinement} and 
Lemma \ref{lem:implies} will imply 
that Theorem \ref{thm:main2} holds.

\section{The proof of Theorem \ref{thm:finalcheck}.}\label{sec:thm}

In this section we will prove Theorem \ref{thm:finalcheck}.  
We will primarily work with the matrices
$\dd$ and $\ee$ defined in Section \ref{sec:sol}, so we begin
by writing down the relations satisfied by these matrices.

\begin{lemma}\label{lem:relations}
Let $\dd$ and $\ee$ be defined as in Section \ref{sec:sol};
let $D$, $E$, $\langle W|$, and $|V\rangle$ be defined 
as in Lemma \ref{thm:USW-solution}; let $A = DE-ED$; 
and let $|V^r\rangle = (0,\dots,0,1,0,0,\dots)^T$, where the $1$
is in the $r$th position, so that $|V^0\rangle = |V\rangle$.
Then we have the following relations.
\begin{align*}
\dd A &= qA\dd\\
A \ee &= q \ee A\\
\dd \ee &= q \ee \dd + (1-q)1 \\
\dd^k \ee &= q^k \ee \dd^k +(1-q^k) \dd^{k-1} \\
\dd |V \rangle &= (b+d) |V\rangle -bd \ee |V\rangle\\
\langle W|\ee &= (a+c) \langle W| -ac \langle W| \dd\\
\dd |V^k \rangle &= d_{k-1}^{\sharp} |V^{k-1}\rangle + 
   d_k^{\natural} | V^k \rangle + d_k^{\flat} |V^{k+1}\rangle.
\end{align*}
\end{lemma}
\begin{proof}
This is a straightforward exercise; most relations follow immediately
from the corresponding relations that $D$ and $E$ satisfy.
\end{proof}

Let $[k]:= 1+q+\dots + q^{k-1}$ be the $q$-analogue of the 
number $k$, and let 
$ \left[\begin{array}{c} N\\ r\end{array}\right]_q 
= \frac{[N]_q !}{[r]_q! [N-r]_q!}.$

\begin{theorem}\label{thm:finalcheck2}
We have that 
\begin{equation}\label{eq:finalcheck2}
\langle W|\dd^N | V^r \rangle = [y^r] \frac{\langle W| (\dd+yA)^N |V \rangle}{\langle W|A^r|V \rangle} = 
\left[\begin{array}{c} N\\ r\end{array}\right]_q 
\frac{\langle W|A^r \dd^{N-r}|V \rangle}{\langle W|A^r |V\rangle}.
\end{equation}
\end{theorem}

\begin{lemma}
Theorem \ref{thm:finalcheck2} implies 
Theorem \ref{thm:finalcheck}.  
\end{lemma}
\begin{proof}
Suppose that Theorem \ref{thm:finalcheck2} is true.   Then
\begin{align*}
\langle W|D^N | V^r \rangle \cdot \tilde \rho_r &= 
  \langle W| \frac{(1+\dd)^N}{(1-q)^N} | V^r \rangle \cdot 
 \alpha^r (1-q)^r \\
&= \alpha^r (1-q)^{r-N} \sum_{i=0}^N {N \choose i} \langle W | \dd^i | V^r \rangle\\
&= \alpha^r (1-q)^{r-N} \sum_{i=0}^N {N \choose i} 
[y^r] \frac{\langle W| (\dd+yA)^i |V \rangle}{\langle W |A^r|V\rangle}\\
& = \frac{\alpha^r (1-q)^{r-N}}{\langle W|A^r|V\rangle}
\sum_{i=0}^N {N \choose i} [y^r] 
\langle W|[(1-q)D-1+yA]^i|V\rangle\\
& = \frac{\alpha^r (1-q)^{r-N}}{\langle W|A^r|V\rangle}[y^r]
\langle W| \sum_{i=0}^N {N \choose i}  
[(1-q)D-1+yA]^i|V\rangle\\
& = \frac{\alpha^r (1-q)^{r-N}}{\langle W|A^r|V\rangle}[y^r]
\langle W|   
[(1-q)D+yA]^N|V\rangle\\
& = \frac{\alpha^r}{\langle W|A^r|V\rangle}[y^r]
\langle W|   
(D+yA)^N|V\rangle\\
& = [y^r]
\frac{(D+y\alpha A)^N|V\rangle}
{\langle W|A^r|V\rangle}.
\end{align*}
\end{proof}

To prove Theorem \ref{thm:finalcheck2}, we will 
find an explicit formula for the quantity in \eqref{eq:finalcheck2}
(see Corollary \ref{cor:1}),
and use recurrences to show that both 
$\langle W|\dd^N | V^r \rangle$ and 
$\left[\begin{array}{c} N\\ r\end{array}\right]_q 
\frac{\langle W|A^r \dd^{N-r}|V \rangle}{\langle W|A^r |V\rangle}$ are equal to this
explicit formula.
Note that the second equality in Theorem \ref{thm:finalcheck2}
follows by repeated application of the relation $\dd A = q A \dd$.

\subsection{A formula for 
${\langle W|A^r \dd^{N-r}|V \rangle}$.}

First we will prove a recurrence for 
${\langle W|A^r \dd^{m}|V \rangle}$ (where we are thinking of 
$m=N-r$).

\begin{proposition}\label{prop:Ak-recurrence}
We have that 
$$
\langle W|A^r \dd^m|V \rangle=
\frac{b+d-bd(a+c)q^{m+r-1}}{1-abcdq^{m+2r-1}} \langle W|A^r \dd^{m-1} |V \rangle
+\frac{bd(q^{m-1}-1)}{1-abcdq^{m+2r-1}} \langle W|A^r \dd^{m-2} |V \rangle.
$$
\end{proposition}

\begin{proof}
In what follows, we will use the relations from Lemma \ref{lem:relations}.
Note that 
\begin{align*}
\langle W|A^r \dd^m|V\rangle &=
 \langle W| A^r \dd^{m-1} \left((b+d)|V\rangle-bd \ee|V\rangle\right)\\
&= (b+d)\langle W|A^r \dd^{m-1}|V\rangle - bd\langle W|A^r \dd^{m-1} \ee|V\rangle\\
&= (b+d)\langle W|A^r \dd^{m-1}|V\rangle - bd \left(q^{m-1} \langle W|A^r \ee \dd^{m-1}|V \rangle -(q^{m-1}-1) \langle W|A^r \dd^{m-2} |V \rangle \right)\\
&= (b+d)\langle W|A^r \dd^{m-1}|V\rangle - bd q^{m-1} \langle W|A^r \ee \dd^{m-1}|V \rangle +bd (q^{m-1}-1) \langle W|A^r \dd^{m-2} |V \rangle \\
&= (b+d)\langle W|A^r \dd^{m-1}|V\rangle - bd q^{m+r-1} \langle W|\ee A^r \dd^{m-1}|V \rangle +bd (q^{m-1}-1) \langle W|A^r \dd^{m-2} |V \rangle \\
&= (b+d)\langle W|A^r \dd^{m-1}|V\rangle - bd q^{m+r-1} \left[
(a+c) \langle W| A^r \dd^{m-1}|V \rangle 
-ac \langle W|\dd A^r \dd^{m-1}|V\rangle \right] \\
& \qquad +bd (q^{m-1}-1) \langle W|A^r \dd^{m-2} |V \rangle \\
&= (b+d)\langle W|A^r \dd^{m-1}|V\rangle - bd (a+c) q^{m+r-1} 
\langle W| A^r \dd^{m-1}|V \rangle +abcd q^{m+2r-1} \langle W|A^r \dd^m|V \rangle \\
&\qquad +bd (q^{m-1}-1) \langle W|A^r \dd^{m-2} |V \rangle. 
\end{align*}
\end{proof}

Our next goal is to get an explicit formula for 
$\frac{\langle W|A^r \dd^{m}|V \rangle}{\langle W|A^r |V\rangle}$
(see Corollary \ref{cor:1}).
Towards this end, we make the following definition.
\begin{definition}\label{def:F}
Define $F_m(y)$ by  
$F_0(y)=1$, $F_m(y)=0$ if $m<0$, and 
\begin{equation}
F_m(y)=(b+d- y(a+c) q^{m-1})F_{m-1}(y))+(q^{m-1}-1)(bd-acq^{m-2}y^2)F_{m-2}(y)
\end{equation}
for $m>0$.
\end{definition}

\begin{corollary}\label{cor:Ak}
We have that 
\begin{equation*}
\frac{\langle W|A^r \dd^m |V \rangle}{\langle W| A^r |V \rangle} = 
\frac{F_m(bd q^r)}{\prod_{i=0}^{m-1} (1-abcd q^{2r+i})}.
\end{equation*}
\end{corollary}
\begin{proof}
Corollary \ref{cor:Ak} follows from 
Proposition \ref{prop:Ak-recurrence} and Definition \ref{def:F}
by comparing the recurrences and base cases.
\end{proof}

\begin{theorem}\label{th:F}
Set
\begin{equation*}
B_m(b,d)=\sum_{i=0}^m 
\left[\begin{array}{c} m\\ i\end{array}\right]_q 
b^i d^{m-i} \text{ and }
A_m(a,c)=\sum_{i=0}^m \left[\begin{array}{c} m\\ i\end{array}\right]_{1/q} a^i c^{m-i}.
\end{equation*}
Then for $m\ge 0$, we have 
\begin{equation}
F_m(y)=\sum_{i=0}^m(-1)^i \left[\begin{array}{c} m\\ i\end{array}\right]_{q}q^{i\choose 2}y^i A_i(a,c)B_{m-i}(b,d).
\end{equation}
\end{theorem}

Before proving Theorem \ref{th:F}, we first state a few
useful lemmas.
Lemma \ref{lem:binomial} is a simple exercise (and is well-known).
\begin{lemma}\label{lem:binomial}
\begin{eqnarray}
 \left[\begin{array}{c} m\\ i\end{array}\right]_q&=& \left[\begin{array}{c} m-1\\ i\end{array}\right]_q+q^{m-i} \left[\begin{array}{c} m-1\\ i-1\end{array}\right]_q\\
 \left[\begin{array}{c} m\\ i\end{array}\right]_q&=& q^i \left[\begin{array}{c} m-1\\ i\end{array}\right]_q+ 
\left[\begin{array}{c} m-1\\ i-1\end{array}\right]_q\\
(1-q^m) \left[\begin{array}{c} m-1\\ i\end{array}\right]_q&=& (1-q^{m-i})\left[\begin{array}{c} m\\ i\end{array}\right]_q
\end{eqnarray}
\end{lemma}

\begin{lemma}\label{lem:AB}
We have that 
\begin{align*}
(b+d)B_m(b,d)&=B_{m+1}(b,d)+(1-q^m)bd B_{m-1}(b,d), \text{ and }\\
(a+c)A_m(a,c)&=A_{m+1}(a,c)+(1-q^{-m})ac A_{m-1}(a,c).
\end{align*}
\end{lemma}
\begin{proof}
It suffices to prove the first equation.  We will use 
Lemma \ref{lem:binomial}.  Note that 
\begin{align*}
(b+d)B_m(b,d)&= \sum_{i=0}^m \left[\begin{array}{c} m\\i\end{array}\right]_q b^i d^{m+1-i} + \sum_{i=1}^{m+1} \left[\begin{array}{c} m\\i-1\end{array}\right]_q b^i d^{m+1-i}\\
&= d^{m+1}+b^{m+1} + 
\sum_{i=1}^m \left( 
\left[\begin{array}{c} m\\i\end{array}\right]_q+
\left[\begin{array}{c} m\\i-1\end{array}\right]_q \right) b^i d^{m+1-i}\\
&= d^{m+1}+b^{m+1} + 
\sum_{i=1}^m \left( 
\left[\begin{array}{c} m+1\\i\end{array}\right]_q+ (1-q^{m-i+1})
\left[\begin{array}{c} m\\i-1\end{array}\right]_q \right) b^i d^{m+1-i}\\
&=  
\sum_{i=0}^{m+1} 
\left[\begin{array}{c} m+1\\i\end{array}\right]_q b^i d^{m+1-i} 
+ \sum_{j=1}^m (1-q^{m-j+1})
\left[\begin{array}{c} m\\j-1\end{array}\right]_q  b^j d^{m+1-j}\\
&=  
\sum_{i=0}^{m+1} 
\left[\begin{array}{c} m+1\\i\end{array}\right]_q b^i d^{m+1-i} 
+ \sum_{i=0}^{m-1} (1-q^{m-i})
\left[\begin{array}{c} m\\i\end{array}\right]_q  b^{i+1} d^{m-i}\\
&=  
\sum_{i=0}^{m+1} 
\left[\begin{array}{c} m+1\\i\end{array}\right]_q b^i d^{m+1-i} 
+ (1-q^m) \sum_{i=0}^{m-1} 
\left[\begin{array}{c} m-1\\i\end{array}\right]_q  b^{i+1} d^{m-i}\\
&= B_{m+1}(b,d)+(1-q^m) bd B_{m-1}(b,d).
\end{align*}
\end{proof}

We now prove Theorem \ref{th:F}.
\begin{proof}
The case $m= 0$ is trivial. We prove the proposition by induction.
Suppose that the proposition is true for all $n<m$. Then
we have the following (see the end of the list of equations
for justifications of each line):
\begin{eqnarray*}
(b+d&-&y(a+c)q^{m-1})F_{m-1}(y)\\
&=&(b+d-y(a+c)q^{m-1}) 
  \sum_{i=0}^{m-1} (-1)^i
  \left[\begin{array}{c} m-1\\ i\end{array}\right]_q q^{i \choose 2} y^i
  A_i(a,c) B_{m-1-i}(b,d)\\
&=& 
  \sum_{i=0}^{m-1} (-1)^i
  \left[\begin{array}{c} m-1\\ i\end{array}\right]_q q^{i \choose 2} y^i
  A_i(a,c) (b+d) B_{m-1-i}(b,d) \\
  &&-\sum_{i=0}^{m-1} (-1)^i
  \left[\begin{array}{c} m-1\\ i\end{array}\right]_q q^{{i \choose 2}+m-1} 
  y^{i+1}
  (a+c) A_i(a,c) B_{m-1-i}(b,d) \\
&=&\sum_{i=0}^m(-1)^i \left[\begin{array}{c} m-1\\ i\end{array}\right]_{q}q^{i\choose 2}y^i A_i(a,c)B_{m-i}(b,d)\\
&&+\sum_{i}(-1)^i \left[\begin{array}{c} m-1\\ i\end{array}\right]_{q}q^{i\choose 2}y^i A_i(a,c)(1-q^{m-i-1})bdB_{m-i-2}(b,d)\\
&&-\sum_{i=0}^m(-1)^i q^{m-1}\left[\begin{array}{c} m-1\\ i\end{array}\right]_{q}q^{i\choose 2}y^{i+1} A_{i+1}(a,c)B_{m-i-1}(b,d)\\
&&-\sum_{i}(-1)^i q^{m-1}\left[\begin{array}{c} m-1\\ i\end{array}\right]_{q}q^{{i\choose 2}}y^{i+1} ac(1-q^{-i})A_{i-1}(a,c)
B_{m-i-1}(b,d)\\
&=&\sum_{i=0}^m(-1)^i \left[\begin{array}{c} m-1\\ i\end{array}\right]_{q}q^{i\choose 2}y^i A_i(a,c)B_{m-i}(b,d)\\
&&+\sum_{i}(-1)^i \left[\begin{array}{c} m-2\\ i\end{array}\right]_{q}q^{i\choose 2}y^i A_i(a,c)(1-q^{m-1})bdB_{m-i-2}(b,d)\\
&&-\sum_{j=1}^m (-1)^{j-1} q^{m-1} \left[\begin{array}{c} m-1 \\ j-1\end{array}\right]_q q^{j-1 \choose 2} y^j A_j(a,c) B_{m-j}(b,d)\\
&&+\sum_{i}(-1)^i q^{m-1-i}\left[\begin{array}{c} m-2\\ i-1\end{array}\right]_{q}q^{{i\choose 2}}y^{i+1} ac(1-q^{m-1})A_{i-1}(a,c)B_{m-i-1}(b,d)\\
&=&\sum_{i=0}^m(-1)^i \left(\left[\begin{array}{c} m-1\\ i\end{array}\right]_{q}+q^{m-i}\left[\begin{array}{c} m-1\\ i-1\end{array}\right]_{q}\right)q^{i\choose 2}y^i A_i(a,c)B_{m-i}(b,d)\\
&&+\sum_{i}(-1)^i \left[\begin{array}{c} m-2\\ i\end{array}\right]_{q}q^{i\choose 2}y^i A_i(a,c)(1-q^{m-1})bdB_{m-i-2}(b,d)\\
&&+\sum_{i}(-1)^i q^{m-1-i}\left[\begin{array}{c} m-2\\ i-1\end{array}\right]_{q}q^{{i\choose 2}}y^{i+1} ac(1-q^{m-1})A_{i-1}(a,c)B_{m-i-1}(b,d)\\
\end{eqnarray*}
\begin{eqnarray*}
&=&\sum_{i=0}^m(-1)^i \left[\begin{array}{c} m\\ i\end{array}\right]_{q}q^{{i\choose 2}}y^i A_i(a,c)B_{m-i}(b,d)\\
&&+(1-q^{m-1})bd \sum_{i}(-1)^i \left[\begin{array}{c} m-2\\ i\end{array}\right]_{q}q^{i\choose 2}y^i A_i(a,c)B_{m-i-2}(b,d)\\
&&-(1-q^{m-1})acy^2q^{m-2} \sum_{j}(-1)^j \left[\begin{array}{c} m-2\\ j\end{array}\right]_{q}q^{j\choose 2}y^j A_j(a,c)B_{m-j-2}(b,d)\\
&=&\sum_{i=0}^m(-1)^i \left[\begin{array}{c} m\\ i\end{array}\right]_{q}q^{i\choose 2}y^i A_i(a,c)B_{m-i}(b,d)
+(1-q^{m-1})(bd-acq^{m-2}y^2)F_{m-2}(y)
\end{eqnarray*}
For the first equality, we used the induction hypothesis,
and for the second equality we simply distributed terms.
For the third equality we used Lemma \ref{lem:AB}. 
For the fourth equality, we changed the index of summation in the 
third line, and we also 
used Lemma \ref{lem:binomial}: 
$$
(1-q^{m-i-1})\left[\begin{array}{c} m-1\\ i\end{array}\right]_{q}=(1-q^{m-1})\left[\begin{array}{c} m-2\\ i\end{array}\right]_{q}
$$
and
$$
(1-q^{i})\left[\begin{array}{c} m-1\\ i\end{array}\right]_{q}=(1-q^{m-1})\left[\begin{array}{c} m-2\\ i-1\end{array}\right]_{q}.
$$
For the fifth equality, we combined two terms.
For the sixth equality, we used Lemma \ref{lem:AB}, and we changed 
the index of summation in the third line.
For the seventh equality, we used the induction hypothesis.

Now we have shown that 
\begin{equation*}(b+d-y(a+c)q^{m-1})F_{m-1}(y)\end{equation*}
is equal to 
\begin{equation*}
\sum_{i=0}^m(-1)^i \left[\begin{array}{c} m\\ i\end{array}\right]_{q}q^{i\choose 2}y^i A_i(a,c)B_{m-i}(b,d)
+(1-q^{m-1})(bd-acq^{m-2}y^2)F_{m-2}(y).
\end{equation*}  
Comparing this to the defining recurrence for $F_m(y)$
(Definition \ref{def:F}), we 
have proved the desired explicit formula for $F_m(y)$.
\end{proof}

\begin{corollary}\label{cor:1}
We have that 
\begin{equation*}
\frac{\langle W| A^r d^m |V \rangle}{\langle W|A^r | V \rangle}  = 
\frac{\sum_{i=0}^m(-1)^i 
\left[\begin{array}{c} m\\ i\end{array}\right]_{q}
q^{i\choose 2}(bdq^r)^i A_i(a,c)B_{m-i}(b,d)}{\prod_{i=0}^{m-1}(1-abcd q^{2r+i})}.
\end{equation*}
\end{corollary}
\begin{proof}
This follows from Corollary \ref{cor:Ak} and Theorem \ref{th:F}.
\end{proof}

\subsection{A formula for 
$\langle W|\dd^N | V^r \rangle$.} 

First we will prove a recurrence for 
$\langle W|\dd^N | V^r \rangle$. 
Towards this end, we state a simple lemma, which follows
from the definitions of the tridiagonal matrices
$\dd$ and $\ee$.

\begin{lemma}\label{lem:border}
\begin{equation*}
\dd|V^r \rangle = (1-q^{2r-1}abcd)|V^{r-1} \rangle -bdq^r \ee|V^r \rangle
  +R_r |V^r \rangle,
\end{equation*}
where 
\begin{equation*}
R_r = \frac{q^{r-1}}{1-abcd q^{2r-2}}
\left(q(b+d)-q^{r-1} (b+d) abcd + bd (a+c)(1-q^r)\right).
\end{equation*}
\end{lemma}

\begin{theorem}\label{thm:dd-recurrence}
The quantity $\langle W|\dd^N |V^r\rangle$ satisfies the following
recurrence.
\begin{align*}
(1-abcd q^{r+N-1}) \langle W|\dd^N |V^r \rangle &= 
(1-q^{2r-1} abcd) \langle W|\dd^{N-1}|V^{r-1}\rangle 
-bdq^r (1-q^{N-1})\langle W|\dd^{N-2}|V^r \rangle \\
&+(R_r - bd q^{r+N-1}(a+c)) \langle W|\dd^{N-1}|V^r \rangle.
\end{align*}
\end{theorem}

\begin{proof}
In what follows, we will use Lemma \ref{lem:border} and Lemma 
\ref{lem:relations} -- in particular the 
relations $\dd^k \ee = q^k \ee \dd^k +(1-q^k) \dd^{k-1}$ and 
$\langle W|\ee = (a+c) \langle W| -ac \langle W| \dd$. 

\begin{align*}
\langle W| \dd^N | V^r \rangle =& 
(1-q^{2r-1} abcd) \langle W| \dd^{N-1}|V^{r-1}\rangle
 - bdq^r \langle W| \dd^{N-1} \ee |V^r \rangle + 
 R_r \langle W| \dd^{N-1} | V^r \rangle\\
=& 
(1-q^{2r-1} abcd) \langle W| \dd^{N-1}|V^{r-1}\rangle
 + R_r \langle W| \dd^{N-1} | V^r \rangle\\
 &- bdq^{r+N-1} \langle W| \ee \dd^{N-1} |V^r \rangle 
   -bd q^r (1-q^{N-1}) \langle W| \dd^{N-2} | V^r \rangle\\ 
=& 
(1-q^{2r-1} abcd) \langle W| \dd^{N-1}|V^{r-1}\rangle
 + R_r \langle W| \dd^{N-1} | V^r \rangle
 - bdq^{r+N-1} (a+c) \langle W|  \dd^{N-1} |V^r \rangle  \\
 &+abcd q^{r+N-1} \langle W| \dd^N |V^r \rangle
   -bd q^r (1-q^{N-1}) \langle W| \dd^{N-2} | V^r \rangle.
\end{align*}
Collecting like terms now yields  the recurrence in 
Theorem \ref{thm:dd-recurrence}.
\end{proof}

\begin{theorem}\label{thm:explicit-d}
We have that 
\begin{equation*}
{\langle W|\dd^{m+r} |V^r \rangle} = 
\left[\begin{array}{c} m+r\\ r\end{array}\right]_q 
\frac{F_m(bd q^r)}{\prod_{i=0}^{m-1} (1-abcd q^{2r+i})}.
\end{equation*}
\end{theorem}

Note that once we have proved Theorem \ref{thm:explicit-d},
combining this theorem with Corollary \ref{cor:Ak} will 
prove 
Theorem \ref{thm:finalcheck2} and hence 
Theorem \ref{thm:finalcheck}.

\begin{proposition}\label{prop:recurrences}
Define
$$C(m,r) = \frac{F_m(bd q^r)}{\prod_{i=0}^{m-1} (1-abcd q^{2r+i})}.$$
Then $C(m,r)$ satisfies the recurrence
$$
(1-abcdq^{2r+m-1})C(m,r)=A_1(m,r)C(m,r-1)-A_2(m,r)C(m-2,r)+A_3(m,r)C(m-1,r),
$$
where
\begin{eqnarray*}
A_1&=&\frac{(1-q^{2r-1}abcd)(1-q^r)}{(1-q^{m+r})}\\
A_2&=&\frac{bdq^r(1-q^{m-1})(1-q^m)}{(1-q^{m+r})}\\
A_3&=&\frac{(-bd(a+c)q^{m-1+2r}+R_r)(1-q^m)}{(1-q^{m+r})}.
\end{eqnarray*}
Equivalently, $F_m(bdq^r)$ satisfies the following recurrence.
\begin{align*}\label{F:recurrence}
(1-q^{m+r})(1&-abcdq^{2r-2}) F_m (bd q^r) =
   (1-q^r) (1-abcd q^{2r+m-2})F_m(bdq^{r-1})\\
    &-bdq^r(1-q^{m-1})(1-q^m)(1-abcdq^{2r-2})(1-abcdq^{2r+m-2})F_{m-2}(bdq^r)\\
&+(1-q^m)(-bd(a+c)q^{2r+m-1}+R_r)(1-abcdq^{2r-2})F_{m-1}(bdq^r).
\end{align*}

\end{proposition}

\begin{lemma}
To prove Theorem \ref{thm:explicit-d}, it suffices to prove 
Proposition \ref{prop:recurrences}.
\end{lemma}
\begin{proof}
To prove Theorem \ref{thm:explicit-d}, we need to show that 
\begin{equation}\label{eq:dd}
{\langle W|\dd^{m+r} |V^r \rangle} = 
\left[\begin{array}{c} m+r\\ r\end{array}\right]_q 
C(m,r).
\end{equation}
It is easy to verify this equation for $m+r \leq 1$.  Moreover 
it is straightforward to verify that the recurrences
in Proposition \ref{prop:recurrences} are equivalent to 
the recurrence for $\langle W|\dd^{m+r}|V^r\rangle$ (see Theorem 
\ref{thm:dd-recurrence}), provided that we have 
\eqref{eq:dd}. 
\end{proof}

Therefore to complete the proof of Theorem \ref{thm:explicit-d},
we need to prove Proposition \ref{prop:recurrences}.  In particular,
we will prove the recurrence for $F_m(bdq^r)$, which follows
from Proposition \ref{prop:F2} below when $y=q^r$.

\begin{proposition}\label{prop:F2}
$F_m(bdy)$ satisfies the following recurrence.
\begin{align*}
(1-yq^m)&(1-y^2q^{-2}abcd)F_m(bdy)=(1-y^2q^{m-2}abcd)(1-y)F_m(bdy/q)\\
&-bdy(1-q^m)(1-q^{m-1})(1-abcdy^2q^{m-2})(1-abcdy^2q^{-2})F_{m-2}(bdy)\\
&+(1-q^m)\{ y\left[b+d+q^{-1}bd(a+c) \right]
               -y^2 \left[ q^{-1}(1+q^m)bd(a+c)+q^{-2}(b+d)abcd \right]\\
               &+y^4 \left[q^{m-3} ab^2 cd^2(a+c) \right] \} F_{m-1}(bdy).
\end{align*}
\end{proposition}

We will prove Proposition \ref{prop:F2} by 
taking the coefficient of $y^n$ in the recurrence and showing that it is an identity.
Towards this end, we define
\begin{equation}\label{def:X}
X_{m,n}=(-1)^n\left[\begin{array}{c} m\\ n\end{array}\right]_{q}q^{n\choose 2}b^nd^n A_n(a,c)B_{m-n}(b,d),
\end{equation}
which is the coefficient of $y^n$ in $F_m(bdy)$.

\begin{lemma}\label{lem:coeff}
Taking the coefficient of $y^n$ in the recurrence  of Proposition \ref{prop:F2} yields the following
equation.
\begin{align}\label{eq:coeff}
X_{m,n} -& q^{-2} abcd X_{m,n-2} - q^m X_{m,n-1} + q^{m-2} abcd X_{m,n-3}  \\
 =&q^{-n} X_{m,n} - q^{-n+1} X_{m,n-1} - q^{m-n} abcd X_{m,n-2} + q^{m-n+1} abcd X_{m,n-3} \nonumber \\
 -&(1-q^m)(1-q^{m-1}) bd X_{m-2,n-1} +  q^{-2} (1-q^m)(1-q^{m-1}) ab^2 cd^2 X_{m-2,n-3} \nonumber \\
 +& (1-q^m)(1-q^{m-1}) \left[ q^{m-2} ab^2 cd^2 X_{m-2,n-3} -  q^{m-4}  (abcd)^2 bd X_{m-2,n-5}\right] \nonumber\\
 +&(1-q^m)(b+d) X_{m-1,n-1} + q^{-1} (1-q^m) bd(a+c) X_{m-1,n-1} \nonumber \\
 -&(1-q^m)\left[ q^{-1}(1+q^m) bd(a+c) + q^{-2} abcd(b+d) \right] X_{m-1,n-2}\nonumber \\
 +& (1-q^m) q^{m-3} ac (a+c) (bd)^2  X_{m-1,n-4}. \nonumber
\end{align}
\end{lemma}
\begin{proof}
The proof is straightforward so we omit it.
\end{proof}

Note that the total degree in $a$ and $c$ in $X_{m,n}$ is $n$, while the 
total degree in $b$ and $d$ in $X_{m,n}$ is $m+n$.
Moreover, each term in \eqref{eq:coeff} has one of the following properties:
\begin{itemize}
\item 
the total degree in $a$ and $c$  is $n$, and the total degree in $b$ and $d$ is $m+n$, or 
\item 
the total degree in $a$ and $c$  is $n-1$, and the total degree in $b$ and $d$ is $m+n-1$.
\end{itemize}
So we can split \eqref{eq:coeff} into two equations based on which of these two conditions holds.
We get equations \eqref{eq:1} and \eqref{eq:2}, respectively.

\begin{align}\label{eq:1}
q^2(1-q^{-n})X_{m,n}-abcd(1-q^{m-n+2})&X_{m,n-2}-q(1-q^m)bd(a+c)X_{m-1,n-1} \\
                         +&(1-q^m)abcd(b+d)X_{m-1,n-2}=0. \nonumber
\end{align}

\begin{align}\label{eq:2}
&(q^{-n+1}-q^m)X_{m,n-1}+(q^{m-2}-q^{m-n+1})abcd X_{m,n-3}-(1-q^m)(b+d)X_{m-1,n-1}\\
&+q^{-1}(1-q^m)(1+q^m)bd(a+c)X_{m-1,n-2}-q^{m-3}(1-q^m)ac(a+c)(bd)^2X_{m-1,n-4} \nonumber\\
&+(1-q^m)(1-q^{m-1})bd\{ X_{m-2,n-1}-(q^{-2}+q^{m-2})abcdX_{m-2,n-3}\nonumber\\
 &\qquad \qquad \qquad \qquad \qquad \qquad \qquad \qquad +q^{m-4}(abcd)^2 X_{m-2,n-5}\}=0.
\nonumber
\end{align}

Now it is clear that Proposition \ref{prop:F2}  follows from Lemmas \ref{lem:eq1} and
\ref{lem:eq2} below.
To prove the lemmas, we will take the coefficient of $a^{j}b^{n+i}c^{n-j}d^{m-i}$ in each of 
\eqref{eq:1} and \eqref{eq:2}, obtaining identities involving 
$q$-binomial coefficients which can be proved by hand or with any computer algebra system.

Let $x(m,n,i,j)$ denote the coefficient of $a^{j}b^{n+i}c^{n-j}d^{m+n-i}$ in $X_{m,n}(q)$.
It is easy to check that 
\begin{equation}\label{eq:x}
x(m,n,i,j)=(-1)^n q^{n \choose 2} \left[\begin{array}{c}m\\n\end{array}\right]_q\left[\begin{array}{c}m-n\\i\end{array}\right]_q\left[\begin{array}{c}n\\j\end{array}\right]_{1/q}.
\end{equation}

\begin{lemma}\label{lem:eq1}
Equation \eqref{eq:1} is an identity.
\end{lemma}
\begin{proof}
It suffices to prove that taking the coefficient of 
$a^j b^{n+i} c^{n-j} d^{m-i}$ in \eqref{eq:1} yields an identity.
When we take this coefficient we obtain the following equation.
\begin{align*}
&(1-q^{-n}) x(m, n, i, j)- (q^{-2}-q^{m-n}) x(m, n-2, i+1, j-1)\\
&-q^{-1} (1-q^m)\left[x(m-1, n-1, i, j)+x(m-1, n-1, i, j-1)\right]\\
&+q^{-2} (1-q^m)
\left[x(m-1, n-2, i+1, j-1)+x(m-1, n-2, i, j-1)\right]=0.
\end{align*}
Using \eqref{eq:x} now yields the equation
\begin{align*}
& q^{n\choose 2}(1-q^{-n})\left[\begin{array}{c}m\\n\end{array}\right]_q\left[\begin{array}{c}m-n\\i\end{array}\right]_q\left[\begin{array}{c}n\\j\end{array}\right]_{1/q}\\
 &-q^{n-2\choose 2}(q^{-2}-q^{m-n})\left[\begin{array}{c}m\\n-2\end{array}\right]_q\left[\begin{array}{c}m-n+2\\i+1\end{array}\right]_q\left[\begin{array}{c}n-2\\j-1\end{array}\right]_{1/q}\\
 &+q^{{n-1\choose 2}-1}(1-q^{m})\left[\begin{array}{c}m-1\\n-1\end{array}\right]_q\left[\begin{array}{c}m-n\\i\end{array}\right]_q\left(\left[\begin{array}{c}n-1\\j-1\end{array}\right]_{1/q}+
 \left[\begin{array}{c}n-1\\j\end{array}\right]_{1/q}\right)\\
 &+q^{{n-2\choose 2}-2}(1-q^{m})\left[\begin{array}{c}m-1\\n-2\end{array}\right]_q\left(\left[\begin{array}{c}m-n+1\\i\end{array}\right]_q+\left[\begin{array}{c}m-n+1\\i+1\end{array}\right]_q
 \right)\left[\begin{array}{c}n-2\\j-1\end{array}\right]_{1/q}=0.
 \end{align*}
Simplifying this, we get the equation 
\begin{align*}
& q^{2n-1}(1-q^{-n})-\frac{q^{n-1}(1-q^{m-n+2})(1-q^j)(1-q^{n-j})}{(1-q^{i+1})(1-q^{m-n+1-i})}
 -q^{2n-1}(2-q^{-j}-q^{-n+j})\\
&+q^{2n-1}(1-q^{-j})(1-q^{-n+j})\left(\frac{1}{1-q^{m-n+1-i}}+\frac{1}{1-q^{1+i}}\right)=0,
 \end{align*}
which is trivially true.
\end{proof}

\begin{lemma}\label{lem:eq2}
Equation \eqref{eq:2} is an identity.
\end{lemma}
\begin{proof}
We use the same strategy as in the proof of Lemma \ref{lem:eq1}.
Now we take the coefficient of 
$a^j b^{n-1+i} c^{n-1-j} d^{m-i}$ in \eqref{eq:2}, obtaining
the following equation.
\begin{align*}
&(q^{-n+1}-q^m)x(m, n-1, i, j)+(q^{m-2}-q^{m-n+1})x(m, n-3, i+1, j-1)\\
&+(1-q^m)(1-q^{m-1})\{x(m-2, n-1, i-1, j)-(q^{-2}+q^{m-2})x(m-2, n-3, i, j-1)\\
&+q^{m-4}x(m-2, n-5, i+1, j-2)\}-(1-q^m)\{x(m-1, n-1, i, j)\\
&+x(m-1, n-1, i-1, j)
-\frac{1+q^m}{q}\left[x(m-1, n-2, i, j-1)+x(m-1, n-2, i, j)\right]\\
&+q^{m-3}
\left[x(m-1, n-4, i+1, j-1)+x(m-1, n-4, i+1, j-2)\right]\}=0
\end{align*}

Again it is straightforward to show that this is an identity.
\end{proof}

\begin{remark} As pointed out by one of the referees, the $F_m(y)$ are deformation of Al-Salam Chihara
polynomials $Q_m(x;a,b|q)$. See for example 
\cite[Section 14.8]{KS}. To be precise,
$$
F_m(y)=(bd)^{m/2}Q_m\left(\frac{b+d}{2\sqrt{bd}}     ; \frac{ya}{\sqrt{bd}}, \frac{yc}{\sqrt{bd}}|q\right).
$$
Thanks to this observation, we can write
$$
F_m(ybd)=\frac{(y^2abcd;q)_m}{(ya)^m}\sum_{k=0}^m\frac{(q^{-m};q)_k(yad;q)_k(yab;q)_k}{(y^2abcd;q)_k(q;q)_k}q^k;
$$
with $(a;q)_k=\prod_{i=0}^{k-1}(1-aq^i)$. Using this formula, we could give an alternative proof of 
Proposition \ref{prop:F2}.
\end{remark}

\section{Koornwinder moments at  $\xi = q=1$.}\label{sec:specialize}

Our goal in this section is to 
show that when one evaluates Koornwinder moments 
at $\xi=q=1$, one obtains  a beautiful multiplicative  formula in terms of hook lengths.
This formula is given in 
Theorem \ref{thm:specialize} below.  

To make sense of this specialization, however, we need to be careful.
It is clear from Section \ref{sec:sol} that the 
Uchiyama-Sasamoto-Wadati solution to the 
Matrix Ansatz  has poles when $q=1$, which is related to the fact
that some of our formulas for Koornwinder moments 
-- in particular Theorem \ref{thm:main} -- are only valid
for $q\neq 1$.
However, 
there are other solutions to the Matrix Ansatz which are well-defined
at $q=1$.  We will review one such solution here, 
use it to define the partition
function and Koornwinder moments (which are the same as our previous 
definitions, up to a global scalar factor), and then state and prove
Theorem \ref{thm:specialize}.  

In this section we will abbreviate 
$K_{\lambda}(1; \alpha, \beta, \gamma, \delta; 1)$
by writing $K_{\lambda}$.

\subsection{The Matrix Ansatz at $q=1$}
We give now a solution to 
the Matrix Ansatz at $q=1$, which comes from 
\cite[Section 5]{USW}.
We define tridiagonal matrices $\D$ and $\E$, and vectors
$\llangle W\vert$ and $\vert V\rrangle$ by 
\begin{eqnarray}\label{NewD}
\D=\left[
\begin{array}{cccc}
\DD_0^\natural  & \DD_0^\sharp  & 0             & \cdots\\
\DD_0^\flat     & \DD_1^\natural& \DD_1^\sharp  & {}\\
0               & \DD_1^\flat   & \DD_2^\natural        & \ddots\\
\vdots          & {}            & \ddots        & \ddots
\end{array}
\right] ,
&&\qquad 
\E=\left[
\begin{array}{cccc}
\EE_0^\natural  & \EE_0^\sharp  & 0             & \cdots\\
\EE_0^\flat     & \EE_1^\natural& \EE_1^\sharp  & {}\\
0               & \EE_1^\flat   & \EE_2^\natural& \ddots\\
\vdots          & {}            & \ddots        & \ddots
\end{array}
\right],
\label{eqn:repDE3}
\end{eqnarray}
\begin{eqnarray*}
&&\llangle W\vert =(1,0,0,\cdots ), \qquad
\vert V\rrangle =(1,0,0,\cdots )^T, 
\label{eqn:repWV3}
\end{eqnarray*}
where
\begin{eqnarray*}
&&\DD^\natural_n=
\frac{\alpha+\delta+n(\alpha\beta+2\alpha\delta+\gamma\delta)}
{(\alpha+\gamma)(\beta+\delta)},
\qquad
\EE^\natural_n=
\frac{\beta+\gamma+n(\alpha\beta+2\beta\gamma+\gamma\delta)}
{(\alpha+\gamma)(\beta+\delta)},
\nonumber\\
&&\DD^\sharp_n=
\frac{\alpha}{\alpha+\gamma}
\left[(n+1)(x+n)\right]^{1/2},
\qquad
\EE^\sharp_n=
\frac{\gamma}{\alpha+\gamma}
\left[(n+1)(x+n)\right]^{1/2},
\nonumber\\
&&\DD^\flat_n=
\frac{\delta}{\beta+\delta}
\left[(n+1)(x+n)\right]^{1/2},
\qquad
\EE^\flat_n=
\frac{\beta}{\beta+\delta}
\left[(n+1)(x+n)\right]^{1/2},
\end{eqnarray*}
with
\begin{eqnarray*}
x = \frac{\alpha+\beta+\gamma+\delta}{(\alpha+\gamma)(\beta+\delta)}.
\end{eqnarray*}

We now use the above solution to the Matrix Ansatz to 
define the partition function and Koornwinder moments at $q=\xi = 1$.

We define the following notation:
\begin{equation*}
S = \frac{(\alpha+\gamma)(\beta+\delta)}{\alpha \beta - \gamma \delta}.
\end{equation*}

\begin{definition}
When $q=1$, we define the partition function of the ASEP
by $Z_N = S^N \langle W | (\D+\E)^N | V \rangle.$\footnote{We have
included the constant $S^N$ in the definition of the partition function
so that this definition of $Z_N$ equals the specialization
of the fugacity partition function $Z_N(\xi)$ from Definition \ref{def:ZN} at 
$\xi = q=1$.  This follows from \cite[Equation 6.16]{USW} and \cite[Theorem 4.1]{CSSW}.}
And we define the Koornwinder moment $K_{\lambda}$ by
\eqref{def:moment}, after setting $\xi=1$ and using 
the above definition of $Z_N$. 
\end{definition}

Let $\mathcal C = (c_{ij})$ be the tridiagonal matrix 
defined by 
\begin{align*}
c_{i,i+1} &= 1 \\
c_{i,i} &= S(D_i^{\natural}+E_i^{\natural}) = S(x+2i) \\
c_{i,i-1} &= S^2(D_{i-1}^{\flat}+E_{i-1}^{\flat})(D_{i-1}^{\sharp}+E_{i-1}^{\sharp}) =S^2  i (x-1+i).
\end{align*}

\begin{lemma}
We have that $Z_N = 
\langle W| {\mathcal C}^N | V \rangle$.
\end{lemma}

\begin{proof}
Note that for every step from height $i$ to height 
$i+1$ in a Motzkin path,
there must be a corresponding step from height $i+1$ to height $i$
in the Motzkin path.  It follows that 
$\langle W| {\mathcal C}^N | V \rangle = S^N
\langle W| (\D + \E)^N | V \rangle$.
\end{proof}

\begin{proposition}\label{prop:1}
We have that 
$$K_{(N-r,0,0,\dots,0)} = S^{N-r}{N \choose r} \prod_{i=r}^{N-1} (x+i)=\frac{1}{(\alpha\beta-\gamma\delta)^{N-r}}{N\choose r}\prod_{i=r}^{N-1}(\alpha+\beta+\gamma+\delta+i(\alpha+\gamma)(\beta+\delta)),$$
where there are precisely $r$ $0$'s in $(N-r,0,0,\dots,0)$.
\end{proposition}

\begin{proof}
If we apply Theorem \ref{thm:det-Motzkin}
and use the fact that each $c_{i,i+1} = 1$, we get 
\begin{equation}
\mathcal{C}\Motz(N,r)
=  
{\langle W|  \mathcal{C}^N | V^r\rangle} = 
K_{(N-r,0,0,\dots,0)}.
\end{equation}
But now note that there is a simple recurrence for 
partial Motzkin paths: 
\begin{eqnarray*}
\mathcal{C}\Motz(N,r)
&=&  
\mathcal{C}\Motz(N-1,r) c_{r,r} +
\mathcal{C}\Motz(N-1,r-1) c_{r-1,r} + 
\mathcal{C}\Motz(N-1,r+1) c_{r+1,r} \\
&=&  
\mathcal{C}\Motz(N-1,r) S (x+2r) +
\mathcal{C}\Motz(N-1,r-1)  +\\ 
&& 
\mathcal{C}\Motz(N-1,r+1) S^2 (r+1)(x+r). 
\end{eqnarray*}
This implies the result.
\end{proof}

\subsection{Koornwinder moments at $q=\xi = 1$}

We identify each partition $\lambda = (\lambda_1,\dots,\lambda_n)$
with its Young diagram, namely a left-justified array of 
cells with $n$ rows and $\lambda_i$ cells in the $i$th row
for $1 \leq i \leq n$.  For each cell $z$ of the Young diagram 
in position $(i,j)$, the \emph{hook} $H_{\lambda}(z)$
is the set of cells $(a,b)$ such that $a=i$ and 
$b \geq j$, or $a\geq i$ and $b=j$.  The \emph{hook length}
$h_{\lambda}(z)$ is the cardinality of $H_{\lambda}(z)$.

Theorem \ref{thm:specialize} is the main result of this section.
\begin{theorem}\label{thm:specialize}

For any partition $\lambda = (\lambda_1,\dots, \lambda_n)$,
we have 
\begin{equation*}
K_{\lambda} = S^{|\lambda|} \prod_{z\in \lambda}
 (x+h(z)-1) \cdot \prod_{i=1}^{n-1} \prod_{j=i+1}^n 
\frac{(x+\lambda_i-\lambda_j+j-i-1)(\lambda_i-\lambda_j+j-i)}
{(x+j-i-1)(j-i)}.
\end{equation*}
\label{qegal1}
\end{theorem}

To prove Theorem \ref{thm:specialize},
we need a few lemmas.

\begin{lemma}\label{Lem:1}
Let $\lambda = (\lambda_1,\dots,\lambda_m)$ and 
$\nu = (\lambda_1,\dots,\lambda_m,0)$.  Then 
\begin{equation*}
K_{\nu} = K_{\lambda} \cdot \prod_{i=1}^m \frac{(x+\lambda_i+m-i)(\lambda_i+m-i+1)}{(x+i-1)(i)}.
\end{equation*}
\label{lemnu}
\end{lemma}

\begin{proof}
By Corollary \ref{thm:JT} and Proposition \ref{prop:1}, we have that 
$$K_{\nu} = \det(M_{ij})_{i,j=1}^{m+1},$$ where 
$$M_{ij} = S^{\nu_i+j-i} {\nu_i+m+1-i \choose m+1-j}
\prod_{\ell = m+1-j}^{\nu_i+m-i} (x+\ell).$$
Note that 
$$M_{ij} = S^{\lambda_i+j-i} {\lambda_i+m+1-i \choose m+1-j}
\prod_{\ell = m+1-j}^{\lambda_i+m-i} (x+\ell)$$ for $i \neq m+1$,
and 
$$M_{m+1,j} = S^{j-m+1} {0 \choose m+1-j} = 
\begin{cases}
1 \hspace{.5cm} \text{ if } j=m+1\\
0 \hspace{.5cm} \text{ otherwise.}
\end{cases} 
$$
Therefore $$\det(M_{ij})_{i,j=1}^{m+1} = \det(M_{ij})_{i,j=1}^m.$$

Meanwhile we have that 
$$K_{\lambda} = \det(M'_{ij})_{i,j=1}^{m},$$ where 
$$M'_{ij} = S^{\lambda_i+j-i} {\lambda_i+m-i \choose m-j}
\prod_{\ell = m-j}^{\lambda_i+m-i-1} (x+\ell).$$

Clearly $M'_{ij} = \frac{(m+1-j)(x+m-j)}{(\lambda_i+m+1-i)(x+\lambda_i+m-i)}
\cdot M_{ij}.$
Therefore 
\begin{align*}
K_{\lambda} &= \det\left(\frac{(m+1-j)(x+m-j)}{(\lambda_i+m+1-i)(x+\lambda_i+m-i)}
\cdot M_{ij} \right)_{i,j=1}^m \\
&= \frac{ \prod_{j=1}^m (m+1-j)(x+m-j)}{\prod_{i=1}^m (\lambda_i+m+1-i)(x+\lambda_i+m-i)} \det(M_{ij})_{i,j=1}^m\\
&= \frac{ \prod_{j=1}^m (j)(x+j-1)}{\prod_{i=1}^m (\lambda_i+m+1-i)(x+\lambda_i+m-i)} K_{\nu}.
\end{align*}
This proves the lemma.
\end{proof}

\begin{lemma}\label{Lem:2}
Let $\lambda = (\lambda_1,\dots,\lambda_n)$ and 
$\nu = (\lambda_1+1,\dots,\lambda_n+1)$.  Then 
\begin{equation*}
K_{\nu} = K_{\lambda} S^n \cdot \prod_{i=1}^n (x+\lambda_i+n-i).
\end{equation*}
\label{lemmu}
\end{lemma}

\begin{proof}
By Corollary \ref{thm:JT} and Proposition \ref{prop:1}, we have that 
$$K_{\nu} = \det(M_{ij})_{i,j=1}^{n},$$ where 
$$M_{ij} = S^{\lambda_i+1+j-i} {\lambda_i+1+n-i \choose n-j}
\prod_{\ell = n-j}^{\lambda_i+n-i} (x+\ell).$$

Meanwhile we have that 
$$K_{\lambda} = \det(M'_{ij})_{i,j=1}^{n},$$ where 
$$M'_{ij} = S^{\lambda_i+j-i} {\lambda_i+n-i \choose n-j}
\prod_{\ell = n-j}^{\lambda_i+n-i-1} (x+\ell).$$

We can now write 
\begin{align*}
M_{ij} &= S^{\lambda_i+1+j-i} \left( {\lambda_i+n-i \choose n-j} + 
{\lambda_i+n-i \choose n-j-1} \right)
\prod_{\ell = n-j}^{\lambda_i+n-i} (x+\ell) \\
 &= S M'_{ij} (x+\lambda_i + n-i) + M'_{i,j+1} \frac{x+\lambda_i+n-i}{x+n-j-1}.
\end{align*}

But now since the determinant is alternating in the columns of 
the matrix $(M_{ij})$, we have that 
\begin{align*}
K_{\nu} &= \det (M_{ij}) = 
\det(S M'_{ij} (x+\lambda_i+n-i))_{i,j=1}^n  \\
 &= S^n 
\det (M'_{ij})_{i,j=1}^n
\prod_{i=1}^n (x+\lambda_i+n-i)\\ 
&= S^n 
K_{\lambda}
\prod_{i=1}^n (x+\lambda_i+n-i). 
\end{align*}
This proves the lemma.
\end{proof}

\begin{proof} [Proof of Theorem \ref{qegal1}.] 
We start with a partition $\lambda=(\lambda_1,\dots,\lambda_n)$. We will
use induction on both $n$ and $\lambda_n$.  If $n=0$ then $K_\lambda=1$ and we are done.
Otherwise if $\lambda_n=0$, we apply Lemma \ref{lemnu} and get
\begin{eqnarray*}
K_\lambda&= & K_{(\lambda_1,\ldots ,\lambda_{n-1})} \cdot \prod_{i=1}^{n-1} \frac{(x+\lambda_i+n-1-i)(\lambda_i+n-i)}{(x+i-1)(i)}\\
&=&
K_{(\lambda_1,\ldots ,\lambda_{n-1})} \cdot \prod_{i=1}^{n-1}
\frac{(x+\lambda_i-\lambda_n+n-1-i)(\lambda_i-\lambda_n+n-i)}
{(x+n-i-1)(n-i)}
\end{eqnarray*}
By induction, we know that
$$
 K_{(\lambda_1,\ldots ,\lambda_{n-1})}= S^{|\lambda|} \prod_{z\in \lambda}
 (x+h(z)-1) \cdot \prod_{i=1}^{n-2} \prod_{j=i+1}^{n-1} 
\frac{(x+\lambda_i-\lambda_j+j-i-1)(\lambda_i-\lambda_j+j-i)}
{(x+j-i-1)(j-i)}
$$
and therefore
$$
K_\lambda=
 S^{|\lambda|} \prod_{z\in \lambda}
 (x+h(z)-1) \cdot \prod_{i=1}^{n-1} \prod_{j=i+1}^n 
\frac{(x+\lambda_i-\lambda_j+j-i-1)(\lambda_i-\lambda_j+j-i)}
{(x+j-i-1)(j-i)}.
$$
Finally if $\lambda_n>0$, we apply Lemma \ref{lemmu},
$$
K_{\lambda} = K_{(\lambda_1-1,\ldots, \lambda_n-1)}\cdot S^n \cdot \prod_{i=1}^n (x+\lambda_i+n-i).
$$
It is easy to see that 
$$
\prod_{z\in \lambda}
 (x+h(z)-1) = \prod_{i=1}^n (x+\lambda_i+n-i) \cdot \prod_{z\in (\lambda_1-1,\ldots ,\lambda_n-1)}
 (x+h(z)-1). 
$$
By induction, we know that the theorem holds for 
$
 K_{(\lambda_1-1,\ldots ,\lambda_n-1)}$,
and therefore we get the desired result for $K_\lambda$.
\end{proof}

We now explain why Conjecture \ref{conj:pos} is true when $\xi = q=1$.  First note that 
by iterating Lemmas \ref{Lem:1} and \ref{Lem:2}, we can prove the following result.

\begin{proposition}\label{prop:recurrence}
Let $\nu = (\lambda_1+1, \dots, \lambda_m+1, 0,\dots, 0)$ and 
$\lambda = (\lambda_1,\dots, \lambda_m,0,\dots,0)$, where there are $r$ $0$'s at the end of each partition.
Then $$K_{\nu} = K_{\lambda} \cdot \prod_{i=1}^m \frac{\lambda_i+m+r+1-i}{\lambda_i+m+1-i} (x+\lambda_i+m+r-i).$$
\end{proposition}

\begin{corollary}
Conjecture \ref{conj:pos} is true when $\xi = q=1$, namely the Koornwinder moments
at this specialization are polynomials in $\alpha, \beta, \gamma, \delta$ with positive coefficients.
\end{corollary}

\begin{proof}
Note that we can repeatedly apply Proposition \ref{prop:recurrence} to an arbitrary partition
$(\lambda_1,\dots, \lambda_n)$, 
so as to express $K_{\lambda}$ in terms of $K_{(0,0,\dots,0)}$.  
It follows from the definition that $K_{(0,0,\dots,0)}$ is equal to $1$, and applying
the proposition preserves the property of being a polynomial with positive coefficients.

\end{proof}

\bibliographystyle{alpha}
\bibliography{bibliography}

\end{document}